\pdfoutput=1
\documentclass[11pt]{article}

\usepackage[wide]{preprint-layout}
\usepackage{txfonts}
\usepackage[hidelinks]{hyperref}

\graphicspath{{figures/}}
\SetSymbolFont{stmry}{bold}{U}{stmry}{m}{n}

\newcommand{\bm}[1]{\boldsymbol{#1}}

\newcommand{\bfn}{{\boldsymbol n}}
\newcommand{\bfI}{{\boldsymbol I}}

\newcommand{\dO}{\,\mathrm{d}\Omega}

\newcommand{\bfeps}{\boldsymbol{\varepsilon}}

\newcommand{\onD}{\qquad \text{on }\Gamma^{\mathrm D}}
\newcommand{\onN}{\qquad \text{on }\Gamma^{\mathrm N}}
\newcommand{\jump}[1]{\llbracket #1\rrbracket}
\newcommand{\subrm}[1]{_{\mathrm{#1}}}
\newcommand{\suprm}[1]{^{\mathrm{#1}}}

\title{Structure-Preserving Augmented Lagrangian\\
Finite Element Methods for Cavitation\\
in Reynolds and Stokes Flows}
\author{Peter Hansbo \qquad Mats G. Larson}
\date{}

\begin{document}

\maketitle

\begin{abstract}
In this paper we propose augmented Lagrangian finite element methods for
cavitation in the Reynolds and Stokes models of lubrication, for which the
discrete complementarity conditions hold exactly and the augmentation
parameter drops out of the method. For the Reynolds equation this follows
from nodal quadrature in a mixed piecewise linear method, for which we prove
the classical first-order error estimate. We also determine the computable
stability threshold of a multiplier-free stabilised alternative. For Stokes
flow we show that the constrained scalar is the mechanical pressure when the
deviatoric stress, with zero bulk viscosity, is used, but not with the
customary incompressible stress. We discretise with a jump-stabilised
Crouzeix--Raviart element and piecewise constant pressure, and prove
well-posedness, stability in two and three dimensions, and a first-order
error estimate. We present numerical examples which verify these properties and show
that close pressure profiles in the two models do not imply close cavity
predictions, the latter also being sensitive to the end conditions of the
computational domain.

\end{abstract}



\section{Introduction}
\label{sec:introduction}

Hydrodynamic lubrication separates moving surfaces by a thin fluid film and
thereby reduces friction and wear. When the local pressure falls to the vapour
pressure, however, the film cavitates and the pressure is constrained from
below. The resulting free boundary is especially important for textured
surfaces: shallow textures can be described by the Reynolds equation, whereas
deep pits and grooves may contain recirculating flow for which the thin-film
approximation is no longer reliable \cite{BaCh86,vOdVe03}. A useful numerical
framework should therefore handle the pressure constraint in both Reynolds
and Stokes models while preserving its mechanical interpretation.

The Reynolds cavitation problem has long been treated as a variational inequality,
beginning with the analysis in \cite{Ci77,CaCi83}. We use the Swift--Stieber
condition, which imposes a nonnegative-pressure constraint and yields a
compact obstacle problem. This model provides a direct setting in which to
study the complementarity structure shared by the Reynolds and Stokes
descriptions. Mass-conserving Elrod--Adams models instead introduce a
saturation variable \cite{El81}; their analysis and discretisation have been
developed in, among others, \cite{BeDu89,GiFoDiSt10,Tb17,AuJaBu09}. The
extension of the present framework to mass-conserving cavitation is discussed
in Section~\ref{sec:conclusions}.

Variational formulations of pressure-constrained Stokes cavitation were
introduced by Nilsson and Hansbo \cite{NiHa11}, who also developed an adaptive
finite element method and compared its predictions with Reynolds lubrication.
Gimbel, Hansbo, and Suttmeier \cite{GiHaSu10} analysed an adaptive low-order
scheme, and the augmented Lagrangian framework of \cite{BuHaLa23} treated
Stokes cavitation with a full-gradient viscous form and a Taylor--Hood
discretisation. The present work deals with two linked structural issues: the constitutive stress needed
for the constrained scalar to represent mechanical pressure, and the discrete
velocity--pressure compatibility needed to enforce complementarity exactly.

We address these questions within the augmented Lagrangian treatment of
inequality constraints developed in \cite{ChHi13,BuHaLaSt17,BuHaLa23}, with
related Reynolds formulations given in
\cite{GuStVi17,GuRaStVi18,Gu24,GrPfCo24}. The central discrete mechanism is
compatibility between the constraint and the finite element spaces: nodal
quadrature makes the Reynolds multiplier equation pointwise, and the
Crouzeix--Raviart/piecewise-constant pair makes the Stokes multiplier
equation elementwise.

\paragraph{Contributions.}
\begin{itemize}
\item We show that the constrained scalar is the mechanical pressure for the
deviatoric Newtonian law with zero bulk viscosity, and that under the
customary incompressible stress a cavitated point can be placed under
tension.
\item For the stabilised Crouzeix--Raviart/piecewise-constant pair the
multiplier residual is elementwise constant, so complementarity holds
exactly and the augmentation parameter drops out; the identity
$\nabla\cdot\vec{V}^h=Q^h$ gives nondegenerate pressure coupling.
\item We prove continuous and discrete well-posedness, stability,
consistency and a first-order a priori error estimate for the deviatoric
formulation with Dirichlet, traction, symmetry and pressure-normal-flow
boundary conditions. In three dimensions the stability proof uses a
boundary-localised piecewise linear specialisation of the discrete
trace-free Korn framework of Williams and Hong \cite{WiHo24}.
\item For the Reynolds problem we prove discrete solvability and the
classical energy-norm estimate for a nodal mixed method, and compare it
with a multiplier-free stabilised method whose stability threshold is
computable.
\item Numerically, we compare constitutive laws and discretisations and
show that the Reynolds and Stokes cavity predictions depend strongly on
the end conditions. With end conditions calibrated against flat Couette
flow, the reformation fronts of the two models lie within a few tenths of
a gap height of each other for all tested pits.
\end{itemize}

\paragraph{Outline.}
Section~\ref{sec:reynolds-cavitation} develops the Reynolds formulations and
their analysis. Section~\ref{sec:stokes-cavitation} treats the constitutive,
continuous, and discrete Stokes questions. Section~\ref{sec:nonlinear-solvers}
describes the nonlinear solution strategies, and
Section~\ref{sec:numerical-results} presents verification and model
comparisons. Concluding remarks are given in Section~\ref{sec:conclusions}; the
broken trace-free Korn inequality is proved in
Appendix~\ref{sec:broken-korn-appendix}.


\section{Reynolds Cavitation}
\label{sec:reynolds-cavitation}

In this section we write the Reynolds cavitation problem in augmented
Lagrangian form, analyse a multiplier-free stabilised discretisation and a
nodal mixed one, and prove an error estimate for the latter.

\subsection{Augmented Formulation and Well-Posedness}
\label{subsec:reynolds-augmented}
Consider a thin lubricant with viscosity $\mu$ enclosed between two surfaces $\Gamma_1$ and $\Gamma_2$ in relative motion. We take $\Gamma_1$ to be the stationary surface that carries the shape, the pitted workpiece, and $\Gamma_2$ the plane above it, sliding with velocity $\bm{v} = (V,0,0)$,
the film occupying the gap of thickness $H(x,y)$ between them.

The Reynolds equation can then be written as
\begin{align}
-\nabla\cdot\left(H^3\nabla p\right)={}& -6\mu V \frac{\partial H}{\partial x}
\quad\text{in }\Omega
\label{eq:reynolds-strong}\\
p={}& 0\quad\text{on }\partial\Omega
\label{eq:reynolds-boundary}
\end{align}
where $H(x, y)$ is the local thickness of the lubricant film, and $p$ is the pressure. For the physical reasoning behind this model, see, e.g., \cite{CaCi83,NiHa11b}. We measure pressure relative to a constant cavitation threshold. In these
examples the ambient boundary pressure is identified with that threshold,
an explicit modelling approximation, so both are represented by zero.

The pressure is constrained from below by this threshold, so an additional
condition is $p\geq 0$ in \eqref{eq:reynolds-strong}. In order to incorporate this
condition into the model, we can introduce a Lagrange multiplier $\lambda$
as follows. Let $c$ be a typical thickness of the film and set
\begin{equation}
P := \frac{pc^2}{6\mu V},\quad d:=\frac{H}{c}, \quad f:= -\frac{\partial d}{\partial x}
\label{eq:reynolds-scaling}
\end{equation}

Then we seek $P$ and $\lambda$ such that
\begin{align}
-\nabla\cdot\left(d^3\nabla P\right)-\lambda ={}& f
\quad\text{in }\Omega
\label{eq:reynolds-kkt-equation}\\
P ={}& 0\quad\text{on }\partial\Omega
\label{eq:reynolds-kkt-boundary}\\
P \geq {}& 0\quad\text{in }\Omega
\label{eq:reynolds-pressure-bound}\\
\lambda \geq {}& 0\quad\text{in }\Omega
\label{eq:reynolds-multiplier-bound}\\
\lambda P =  {}& 0\quad\text{in }\Omega
\label{eq:reynolds-complementarity}
\end{align}
where the last three conditions are the well known Kuhn--Tucker conditions. We now follow \cite{ChHi13,BuHaLaSt17,BuHaLa23} and rewrite the Kuhn--Tucker conditions as
\begin{equation}\label{eq:reynolds-complementarity-map}
\lambda = \frac{1}{\gamma}\left[\gamma\lambda - P\right]_+
\end{equation}
where $[x]_+ = \max(0, x)$ and $\gamma\in {\mathbb{R}}^+$ but otherwise
arbitrary.

The equivalence of \eqref{eq:reynolds-complementarity-map} with the
Kuhn--Tucker conditions rests on the following elementary observation.

\begin{lem}\label{lem:complementarity-map}
Let $\gamma>0$ and $a,b\in{\mathbb{R}}$. Then
\begin{equation}
a = \frac{1}{\gamma}\left[\gamma a - b\right]_+
\qquad\Longleftrightarrow\qquad
a\geq 0,\quad b\geq 0,\quad ab = 0
\label{eq:scalar-complementarity-equivalence}
\end{equation}
\end{lem}

\begin{proof}
The right-hand side is nonnegative, so $a<0$ is impossible. For $a>0$ the
identity holds if and only if $\gamma a-b=\gamma a$, that is, $b=0$, and for
$a=0$ it holds if and only if $[-b]_+=0$, that is, $b\geq 0$.
\end{proof}

Now, following the development for the obstacle problem in \cite{BuHaLaSt17}, we introduce
the augmented Lagrangian problem of finding $(P,\lambda)$ that are stationary points to the functional
\begin{equation}
\begin{aligned}
\mathfrak{F}(P,\lambda) := {}& \frac12\int_{\Omega} d^3\vert\nabla P\vert^2\dO
- \int_{\Omega} f P \dO \\
&+ \int_{\Omega}\frac{1}{2\gamma}\left[\gamma\lambda-P\right]_+^2\dO
-\frac12\int_{\Omega} \gamma \lambda^2 \dO
\end{aligned}
\label{eq:reynolds-augmented-functional}
\end{equation}
leading to seeking $(P,\lambda)\in H_0^1(\Omega)\times L_2(\Omega)$ such that
\begin{equation}
 \int_{\Omega}d^3 \nabla P\cdot\nabla q\dO - \int_{\Omega}\frac{1}{\gamma}\left[\gamma\lambda-P\right]_+ q\dO =\int_{\Omega} f q \dO\quad\forall q\in H_0^1(\Omega)
\label{eq:reynolds-augmented-primal}\end{equation}
and
\begin{equation}
\int_{\Omega}\frac{1}{\gamma}\left[\gamma\lambda-P\right]_+ \eta\dO -\int_{\Omega} \lambda \eta  \dO=0\quad\forall \eta\in L_2(\Omega)
\label{eq:reynolds-augmented-multiplier}\end{equation}

Because \eqref{eq:reynolds-augmented-multiplier} holds for every $\eta\in L_2(\Omega)$, its integrand
vanishes almost everywhere, so Lemma~\ref{lem:complementarity-map} applies pointwise and the
pair \eqref{eq:reynolds-augmented-primal}--\eqref{eq:reynolds-augmented-multiplier} is precisely the Kuhn--Tucker system.
This gives the following, in which
$K := \{v\in H^1_0(\Omega):\ v\geq 0 \text{ a.e.\ in }\Omega\}$ and
$a(P,q) := (d^3\nabla P,\nabla q)_\Omega$.

\begin{thm}\label{thm:reynolds-well-posedness}
Let $d^3\in L_\infty(\Omega)$ with $d^3\geq d_0^3>0$ and let
$f\in L_2(\Omega)$. Then the obstacle problem of finding $P\in K$ with
\begin{equation}\label{eq:reynolds-variational-inequality}
a(P,v-P)\geq (f,v-P)_\Omega\qquad\forall v\in K
\end{equation}
has exactly one solution. If moreover
$\lambda := -\nabla\cdot(d^3\nabla P)-f$ belongs to $L_2(\Omega)$, then for
every $\gamma>0$ the pair $(P,\lambda)$ solves
\eqref{eq:reynolds-augmented-primal}--\eqref{eq:reynolds-augmented-multiplier}, and conversely every solution of
\eqref{eq:reynolds-augmented-primal}--\eqref{eq:reynolds-augmented-multiplier} is of this form. In particular the
solution set of \eqref{eq:reynolds-augmented-primal}--\eqref{eq:reynolds-augmented-multiplier} does not depend on
$\gamma$.
\end{thm}

\begin{proof}
The form $a(\cdot,\cdot)$ is bounded and, by the Poincar\'e inequality and
$d^3\geq d_0^3$, coercive on $H^1_0(\Omega)$, and $K$ is nonempty, closed
and convex, so existence and uniqueness is the theorem of Lions and
Stampacchia \cite{KiSt80}. If $(P,\lambda)$ solves
\eqref{eq:reynolds-augmented-primal}--\eqref{eq:reynolds-augmented-multiplier},
Lemma~\ref{lem:complementarity-map} gives $P\geq 0$, $\lambda\geq 0$ and
$\lambda P = 0$ almost everywhere, and \eqref{eq:reynolds-augmented-primal}
reduces to
\begin{equation}\label{eq:reynolds-kkt-equilibrium}
a(P,q)-(\lambda,q)_\Omega = (f,q)_\Omega\qquad\forall q\in H^1_0(\Omega)
\end{equation}
so that $a(P,v-P)-(f,v-P)_\Omega=(\lambda,v)_\Omega\geq 0$ for $v\in K$.
Conversely, if $P$ solves \eqref{eq:reynolds-variational-inequality} and
$\lambda := -\nabla\cdot(d^3\nabla P)-f\in L_2(\Omega)$, taking $v=P+\varphi$
with $\varphi\in K$ gives $\lambda\geq 0$, and $v=0$ and $v=2P$ give
$(\lambda,P)_\Omega=0$, hence $\lambda P=0$ almost everywhere;
Lemma~\ref{lem:complementarity-map} then returns
\eqref{eq:reynolds-augmented-multiplier}. Neither step involves $\gamma$.
\end{proof}

\begin{rem}
The requirement $\lambda\in L_2(\Omega)$ is a regularity assumption on the
obstacle problem. If $f\in L_2(\Omega)$ and $d^3\in W^{1,\infty}(\Omega)$
is uniformly positive, interior regularity gives
$\lambda\in L_{2,{\rm loc}}(\Omega)$; see \cite{KiSt80}. Global
$L_2$ regularity of $\lambda$, and the $H^2$ regularity used in
Theorem~\ref{thm:reynolds-error}, require a corresponding global obstacle
regularity result, including assumptions on the domain, coefficient, data,
and compatibility at the boundary of the contact set. We therefore state
$P\in H^2(\Omega)$ and $\lambda\in L_2(\Omega)$ explicitly whenever those
properties are used. They allow the multiplier to be sought in
$L_2(\Omega)$ rather than only in $H^{-1}(\Omega)$.
\end{rem}

\subsection{Multiplier-Free Stabilised Method}
\label{subsec:reynolds-stabilized}

For our discrete method, we assume that $\{\mathcal{T}_h\}_{h}$ is a
family of conforming shape-regular meshes on $\Omega$ consisting of
triangles; we write $\mathcal{T} = \{T\}$ for a generic member and
suppress the subscript.
Then we define $V_h$ as the space of $H^1$-conforming piecewise polynomial
functions
on $\mathcal{T}$, satisfying the homogeneous boundary condition on
$\partial\Omega$.
\begin{equation}
V_h := \{v_h \in H^1_0(\Omega): v_h\vert_T \in \mathbb{P}_k(T), \,
\forall T \in \mathcal{T} \},\quad \mbox{ for } k \ge 1
\label{eq:reynolds-finite-element-space}
\end{equation}
To obtain a discrete minimisation problem, without multiplier, we formally replace $\lambda$ elementwise by
$-\nabla\cdot d^3\nabla P_h - f$.  This carries its intended meaning only for
$k\geq 2$: on straight-sided triangles $\Delta v_h$ vanishes identically when
$k=1$, and the substitution then retains only the lower-order part
$3d^2\nabla d\cdot\nabla v_h$ of the residual.  Eliminating the multiplier in
this way is precisely the mechanism by which an augmented Lagrangian
formulation generates a stabilised method, cf.\ \cite{BuHaLa23}.  We thus
seek $P_h\in V_h$ such that
\begin{equation}
P_h = \arg\min_{v\in V_h} \mathfrak{F}_h(v)
\label{eq:reynolds-discrete-minimization}
\end{equation}
where
\begin{equation}
\begin{aligned}
\mathfrak{F}_h(v) := {}& \frac12\int_{\Omega} d^3\vert\nabla v\vert^2\dO
+ \sum_{T\in\mathcal{T}}\int_{T}\frac{1}{2\gamma}
\left[-v-\gamma(\nabla\cdot d^3\nabla v+f)\right]_+^2\dO \\
&-\frac12\sum_{T\in\mathcal{T}}\int_{T} \gamma
\left(\nabla\cdot d^3\nabla v + f\right)^2 \dO
-\int_{\Omega} f v \dO
\end{aligned}
\label{eq:reynolds-discrete-functional}
\end{equation}
The
Euler--Lagrange equations corresponding to \eqref{eq:reynolds-discrete-functional} take the form: Find $P_h \in V_h$ such that
\begin{equation}\label{eq:reynolds-stabilized-method}
a(P_h,v_h) + b(P_h,f;v_h) =(f,v_h)_\Omega \quad \forall v_h
\in V_h
\end{equation}
where $(\cdot,\cdot)_\Omega$ denotes the $L_2$ inner product,
$a(P_h,v_h) := (d^3\nabla P_h,\nabla v_h)_{\Omega}$ and
\begin{equation}
\begin{aligned}
b(P_h,f;v_h):={}& \left<
 \gamma^{-1}[-P_h- \gamma (\nabla\cdot d^3\nabla P_h + f)]_+,
 -v_h - \gamma \nabla\cdot d^3\nabla v_h\right>_{h} \\
 &{} - \left<\gamma \left(\nabla\cdot d^3\nabla P_h + f\right),
 \nabla\cdot d^3\nabla v_h\right>_{h}
\end{aligned}
\label{eq:reynolds-stabilized-form}
\end{equation}
where
\begin{equation}
\left<x_h,y_h\right>_{h} := \sum_{T\in\mathcal{T}} \int_T x_h y_h\,\mathrm{d}x
\label{eq:elementwise-inner-product}
\end{equation}
and, for use below,
\begin{equation}
\| x_h\|_h := \left<x_h,x_h\right>_{h}^{1/2}
\label{eq:elementwise-norm}
\end{equation}
To simplify the notation below we introduce $Q_\gamma(P_h) =
-\gamma \nabla\cdot d^3\nabla P_h-P_h $, in terms of which
\begin{equation}
b(P_h,f;v_h)=\left<
 \gamma^{-1}[Q_\gamma(P_h)- \gamma f]_+, Q_\gamma(v_h) \right>_{h} -   \left<\gamma \left(\nabla\cdot d^3\nabla P_h + f\right), \nabla\cdot d^3\nabla
 v_h\right>_{h}
\label{eq:reynolds-stabilized-q-form}
\end{equation}

In a primal--dual active set (semismooth Newton) iteration
\cite{HiItKu02} one freezes the sign of $Q_\gamma(P_h)-\gamma f$ at each
quadrature point and solves the resulting linear problem. Writing
$\mathcal{A}$ for the set on which the positive part is active and
$\left<\cdot,\cdot\right>_{\mathcal{A}}$ for the elementwise $L_2$ form
restricted to it, the second variation of $\mathfrak{F}_h$ with the sign so
frozen is the symmetric form
\begin{equation}\label{eq:reynolds-frozen-form}
\mathcal{B}_{\mathcal{A}}(v_h,v_h) := a(v_h,v_h)
+ \gamma^{-1}\left<Q_\gamma(v_h),Q_\gamma(v_h)\right>_{\mathcal{A}}
- \gamma\|\nabla\cdot d^3\nabla v_h\|_h^2
\end{equation}
Its last term carries a negative sign, and stability depends on whether it is dominated by
$a(\cdot,\cdot)$.

This was settled for a stabilised method of this kind, applied to the same
inequality constrained Reynolds equation, by Gustafsson, Rajagopal,
Stenberg and Videman \cite[Theorem 2]{GuRaStVi18}, whose condition on the
stabilisation parameter is the one below; we restate it in the notation
used here.

\begin{thm}[{cf.\ \cite[Theorem 2]{GuRaStVi18}}]\label{thm:reynolds-stability}
Let $d\in W^{1,\infty}(\Omega)$ with $d\geq d_0>0$, let $\mathcal{T}$ be
shape regular with $h_T\leq 1$, and let $\gamma\vert_T = \gamma_0h_T^2$.
Then there is a constant $C_I$, depending only on the shape regularity of
$\mathcal{T}$, on the polynomial degree $k$ and on
$\|d\|_{W^{1,\infty}(\Omega)}$, such that
\begin{equation}\label{eq:reynolds-frozen-coercivity}
\mathcal{B}_{\mathcal{A}}(v_h,v_h)\geq
\left(d_0^3-\gamma_0C_I^2\right)\|\nabla v_h\|^2_{\Omega}
\qquad\forall v_h\in V_h
\end{equation}
for every choice of $\mathcal{A}$. If $\gamma_0 < d_0^3C_I^{-2}$ then
$\mathfrak{F}_h$ is strictly convex on $V_h$, so \eqref{eq:reynolds-discrete-functional} has exactly
one solution, and every linear system arising in the active set iteration
is symmetric and positive definite, uniformly in $h$.
\end{thm}

\begin{proof}
Write $Lv := \nabla\cdot(d^3\nabla v) = d^3\Delta v + 3d^2\nabla d\cdot\nabla v$.
The first term obeys the inverse estimate for polynomials and the second is
bounded outright, so that $h_T\|Lv_h\|_T\leq C_I\|\nabla v_h\|_T$ with
$C_I$ depending only on the shape regularity, on $k$ and on
$\|d\|_{W^{1,\infty}}$. Summing over the elements with
$\gamma\vert_T=\gamma_0h_T^2$ bounds the last term of \eqref{eq:reynolds-frozen-form} by
$\gamma_0C_I^2\|\nabla v_h\|^2_\Omega$, the middle term is nonnegative and
$a(v_h,v_h)\geq d_0^3\|\nabla v_h\|^2_\Omega$, which is \eqref{eq:reynolds-frozen-coercivity}. The
remaining term of $\mathfrak{F}_h$ is convex, being the composition of
$[\,\cdot\,]_+^2$ with an affine map, so under the stated bound $\mathfrak{F}_h$
is strictly convex.
\end{proof}

The method is moreover consistent. At the exact
solution $Q_\gamma(P)-\gamma f = \gamma\lambda-P$, so Lemma~\ref{lem:complementarity-map}
gives $\gamma^{-1}[Q_\gamma(P)-\gamma f]_+ = \lambda$, the two terms
carrying $\gamma$ in $b(\cdot,\cdot;\cdot)$ cancel identically, and
\eqref{eq:reynolds-stabilized-method} reduces to \eqref{eq:reynolds-kkt-equilibrium}. An a priori estimate follows;
\cite[Theorem 3]{GuRaStVi18} gives a best approximation result for the
pressure and the multiplier together, in a mesh-dependent norm and under
the same condition on the parameter, with the proof in \cite{GuStVi17}.

\begin{rem}\label{rem:reynolds-parameter-threshold}
The threshold of Theorem~\ref{thm:reynolds-stability} can be located exactly. Since the middle term of
\eqref{eq:reynolds-frozen-form} is nonnegative,
$\mathcal{B}_{\mathcal{A}}\geq\mathcal{B}_\emptyset$ for every
$\mathcal{A}$: the empty set, which is also the first pass of the
iteration, is the worst case, and a threshold computed for it serves for
the whole iteration. Now $\mathcal{B}_\emptyset$ is
positive definite if and only if $\gamma_0<1/\Lambda_{\max}$, where
\begin{equation}
\Lambda_{\max} := \sup_{v_h\in V_h\setminus\{0\}}\
\frac{\sum_{T\in\mathcal{T}}h_T^2\|\nabla\cdot d^3\nabla v_h\|_T^2}
{a(v_h,v_h)} \ \leq\ \frac{C_I^2}{d_0^3}
\label{eq:reynolds-threshold-eigenvalue}
\end{equation}
By \eqref{eq:reynolds-frozen-coercivity}, $\Lambda_{\max}$ is bounded
under refinement. Section~\ref{subsec:reynolds-stabilized-results} computes
it for $\mathbb{P}_2$ and shows that above the resulting threshold the
frozen systems are indefinite and the iteration does not converge. Thus $\gamma$, arbitrary in the continuous formulation by
Theorem~\ref{thm:reynolds-well-posedness}, is in
\eqref{eq:reynolds-discrete-functional} a local, mesh-dependent parameter
with a computable admissible range.
\end{rem}

\subsection{Structure-Preserving Nodal Mixed Method}
\label{subsec:reynolds-nodal-method}

The multiplier need not be eliminated. Returning to
\eqref{eq:reynolds-augmented-primal}--\eqref{eq:reynolds-augmented-multiplier}, we discretise them directly, taking
continuous piecewise linears for both $P$ and $\lambda$ and evaluating the
$L_2$ pairings in which $\lambda$ is tested by nodal quadrature,
$m_i = \int_\Omega \varphi_i\dO$. For $k=1$ this is not a simplification of
the constraint
since a piecewise linear
function attains its extrema at the nodes, so that $P_h\geq 0$ in $\Omega$
is exactly $P_i \geq 0$ at every node. Writing $\mathcal{A}$ for the set of
cavitated nodes and $\mathcal{I}$ for its complement, with ties in the
active-set test assigned to $\mathcal{A}$, the frozen system is
\begin{equation}\label{eq:reynolds-nodal-system}
\left[\begin{array}{cc} A + \gamma^{-1}D_{\mathcal{A}} & -D_{\mathcal{A}} \\
-D_{\mathcal{A}} & -\gamma D_{\mathcal{I}}\end{array}\right]
\left[\begin{array}{c} P \\ \lambda\end{array}\right] =
\left[\begin{array}{c} F \\ 0\end{array}\right]
\end{equation}
with $A$ the $d^3$-weighted stiffness matrix,
$D_{\mathcal{A}} = {\rm diag}(m_i\mathbf{1}_{i\in\mathcal{A}})$, and
$D_{\mathcal{I}} = {\rm diag}(m_i\mathbf{1}_{i\in\mathcal{I}})$. It is symmetric and
has the same block structure as the Stokes system \eqref{eq:stokes-discrete-momentum}--\eqref{eq:stokes-discrete-multiplier},
with $P$ in the role of $\nabla\cdot\bm{u}$ and $\lambda$ in that of $p$.

\begin{thm}\label{thm:reynolds-nodal-structure}
Let $V_h$ consist of continuous piecewise linears vanishing on
$\partial\Omega$ and let $A$ be the corresponding stiffness matrix, which
is symmetric and positive definite. Then, for every choice of
$\mathcal{A}$, the matrix in \eqref{eq:reynolds-nodal-system} is nonsingular. Its
solution satisfies $P_i = 0$ on $\mathcal{A}$ and $\lambda_i = 0$ off it,
so that the product $\lambda_iP_i$ vanishes at every node, for every
$\mathcal{A}$. The test argument $\gamma\lambda_i-P_i$ reduces to
$\gamma\lambda_i$ on $\mathcal{A}$ and to $-P_i$ on $\mathcal{I}$.
At a fixed point, membership therefore gives $\lambda_i\geq0$ on
$\mathcal{A}$ and $P_i>0$ on $\mathcal{I}$, because the test fails there.
Neither the frozen solutions nor the fixed point depend on $\gamma$, and
the fixed point satisfies the discrete Kuhn--Tucker conditions exactly.
Finally, the limit is the unique
minimiser of $\tfrac12a(v_h,v_h)-(f,v_h)_\Omega$ over
$K_h := \{v_h\in V_h:\ v_h\geq 0\ \text{in}\ \Omega\}$, and $K_h\subset K$.
\end{thm}

\begin{proof}
With zero right-hand side the second block row gives $P_i=0$ on
$\mathcal{A}$ and $\lambda_i=0$ on $\mathcal{I}$, since $m_i>0$. The first
row at $i\in\mathcal{I}$ is then $(AP)_i=0$, and since $P$ vanishes on
$\mathcal{A}$, $P^{\rm T}AP=0$, so $P=0$; the first row on $\mathcal{A}$
then gives $\lambda=0$. The same computation with the actual right-hand side
gives the stated structure, and with it the reduction of the test argument
and the independence of $\gamma$. A piecewise linear function attains its
extrema at the nodes, so the nodal constraints describe $K_h$ exactly, and
$K_h\subset K$. At a fixed point the conditions $\lambda_i\geq 0$ and $P_i=0$
on $\mathcal{A}$ and $P_i>0$ and $\lambda_i=0$ on $\mathcal{I}$ are the
Kuhn--Tucker conditions of $\min\{\tfrac12v^{\rm T}Av-F^{\rm T}v:\ v\geq 0\}$,
which are sufficient because $A$ is positive definite.
\end{proof}

\subsection{Error Estimate}
\label{subsec:reynolds-error}

Because $K_h\subset K$ the method is conforming, and an error estimate
follows in the classical way.

\begin{thm}\label{thm:reynolds-error}
Let $P$ and $\lambda$ be as in Theorem~\ref{thm:reynolds-well-posedness} and let $P_h$ solve the
nodal method. Then
\begin{equation}\label{eq:reynolds-quasi-optimal-error}
\|P-P_h\|_{H^1(\Omega)}\leq C\inf_{v_h\in K_h}
\left(\|P-v_h\|_{H^1(\Omega)}
+\|\lambda\|_{L_2(\Omega)}^{1/2}\|P-v_h\|_{L_2(\Omega)}^{1/2}\right)
\end{equation}
with $C$ depending only on $d_0$, on $\|d^3\|_{L_\infty}$ and on $\Omega$.
If in addition $P\in H^2(\Omega)$ then
\begin{equation}\label{eq:reynolds-h1-rate}
\|P-P_h\|_{H^1(\Omega)}\leq C h
\left(\vert P\vert_{H^2(\Omega)}
+\|\lambda\|_{L_2(\Omega)}^{1/2}\vert P\vert_{H^2(\Omega)}^{1/2}\right)
\end{equation}
\end{thm}

\begin{proof}
Let $e:=P-P_h$ and $v_h\in K_h$. By Theorem~\ref{thm:reynolds-nodal-structure},
$a(P_h,v_h-P_h)\geq(f,v_h-P_h)_\Omega$, and by
\eqref{eq:reynolds-kkt-equilibrium}, $a(P,w)-(f,w)_\Omega=(\lambda,w)_\Omega$
for $w\in H^1_0(\Omega)$. Hence
\begin{equation}\label{eq:reynolds-discrete-variational-bound}
a(e,e) = a(e,P-v_h)+a(P,v_h-P_h)-a(P_h,v_h-P_h)
\leq a(e,P-v_h)+(\lambda,v_h-P_h)_\Omega
\end{equation}
Since $(\lambda,P)_\Omega=0$ and $(\lambda,P_h)_\Omega\geq 0$, the last term
is at most $(\lambda,v_h-P)_\Omega\leq\|\lambda\|_{L_2}\|P-v_h\|_{L_2}$, and
coercivity, continuity and Young's inequality give
\eqref{eq:reynolds-quasi-optimal-error}. For \eqref{eq:reynolds-h1-rate} take
$v_h=I_hP$, the Lagrange interpolant, which lies in $K_h$ because
$P\geq 0$ at the nodes, and use $\|P-I_hP\|_{H^1}\leq Ch\vert P\vert_{H^2}$
and $\|P-I_hP\|_{L_2}\leq Ch^2\vert P\vert_{H^2}$.
\end{proof}

\begin{rem}
The rate in \eqref{eq:reynolds-h1-rate} is optimal for piecewise linears.
For higher polynomial degree the corresponding higher rate is not attained
in general, since the solution is at best $C^{1,1}$ across the free
boundary however smooth the data. Section~\ref{subsec:reynolds-convergence}
reports a first-order \(H^1\) rate for the raw finite element solution on a
shape-regular two-dimensional mesh family with a circular free boundary.
We have not proved an \(L_2\) estimate; the second-order pressure error
observed there is specific to that experiment.
\end{rem}


\section{Stokes Cavitation}
\label{sec:stokes-cavitation}

In this section we show that the deviatoric stress makes the constrained
scalar the mechanical pressure, prove well-posedness of the augmented
formulation, and analyse its stabilised Crouzeix--Raviart discretisation,
including an a priori error estimate.

\subsection{Mechanical Pressure and Constitutive Consistency}
\label{subsec:mechanical-pressure}

The problem of this section was treated as an application of the abstract
augmented Lagrangian framework in \cite{BuHaLa23}, using the full gradient
in the viscous form. Here we give a self-contained derivation and show that
the deviatoric law, which has zero bulk viscosity, makes the constrained
scalar equal to the mechanical pressure.

Consider a domain $\Omega$ in ${\mathbb{R}}^n$, $n=2$ or $n=3$ with boundary
$\partial\Omega$. We consider a lubricant with viscosity $\mu$. The Stokes
equation can then be written
\begin{equation}\label{eq:stokes-strong}
-\nabla\cdot\bm{\sigma}(\boldsymbol u,p) = {\boldsymbol f}\;\text{and}\; \nabla\cdot
\boldsymbol u = 0\quad\text{in}\;\Omega
\end{equation}
with the constitutive law
\begin{equation}\label{eq:deviatoric-cauchy-stress}
\bm{\sigma}(\boldsymbol u,p) := 2\mu\left(\bfeps(\boldsymbol u)
-\tfrac13\left(\nabla\cdot\boldsymbol u\right)\bfI\right) - p\,\bfI ,
\qquad
\bfeps(\boldsymbol u) := \frac12\left(\nabla\boldsymbol u + \left(\nabla\boldsymbol u\right)^{\rm T}\right)
\end{equation}
that is, the Newtonian law with the bulk viscosity set to zero, which is
Stokes' hypothesis; see \cite{Ba67}. In Lam\'e form \eqref{eq:deviatoric-cauchy-stress} amounts
to taking the first Lam\'e parameter equal to $-2\mu/3$. The factor
$\tfrac13$ is that of the physical, three-dimensional stress and is kept
also when $n=2$: a two-dimensional computation is a plane flow, its stress
tensor has a third normal component, and the trace below is over the three
physical directions (cf.
Remark~\ref{rem:deviatoric-stability} below).

For an incompressible flow the deviatoric projection in \eqref{eq:deviatoric-cauchy-stress} is
inert and one may as well write $\bm{\sigma} = 2\mu\bfeps(\bm{u})-p\bfI$,
which is what one usually does. Because of the constraint this is not the form we use here.

We adopt, as lubrication theory classically does, a pressure criterion for
cavitation: the film ruptures where the mechanical pressure, the negative
of the mean normal stress,
\begin{equation}\label{eq:mechanical-pressure}
\bar p := -\tfrac13\,{\rm tr}\,\bm{\sigma}
\end{equation}
would fall below the vapour pressure, taken as zero. The trace is over the
three physical directions. Joseph \cite{Jo98} has proposed instead a
criterion on the largest principal tensile stress.

Taking the
trace of \eqref{eq:deviatoric-cauchy-stress} and using ${\rm tr}\,\bfeps^{\rm dev}=0$ gives
$\bar p = p$, so that the constraint imposed below, $p\geq 0$, is exactly
\eqref{eq:mechanical-pressure}. Zero bulk viscosity is thus the constitutive closure that
makes $p$ the mechanical pressure. Had we kept
$\bm{\sigma} = 2\mu\bfeps(\bm{u})-p\bfI$ the same computation would have
returned
\begin{equation}
\bar p = p - \tfrac{2\mu}{3}\,\nabla\cdot\bm{u}
\label{eq:mechanical-pressure-standard-law}
\end{equation}
and imposing $p\geq0$ would constrain a quantity differing from the
mechanical pressure wherever $\nabla\cdot\bm{u}>0$. On the cavitated set,
where $p=0$, it gives $\bar p = -\tfrac{2\mu}{3}\nabla\cdot\bm{u}\leq 0$
under the customary stress, with strict inequality wherever the lubricant
dilates; Section~\ref{subsec:constitutive-discretization} gives numbers.

The two constitutive laws assign stress tensors to the same $(\bm{u},p)$
that differ by the isotropic term $\tfrac{2\mu}{3}(\nabla\cdot\bm{u})\bfI$,
so every principal stress is shifted by that amount, and a criterion on the
largest principal stress is affected in the same way. The distinction is
invisible in the Reynolds model of Section~\ref{sec:reynolds-cavitation}.

\subsection{Augmented Formulation}
\label{subsec:stokes-augmented}

We split the boundary into disjoint relatively open parts,
$\partial\Omega = \overline{\Gamma\suprm{D}}\cup
\overline{\Gamma\suprm{N}}\cup\overline{\Gamma\suprm{S}}\cup\overline{\Gamma\suprm{E}}$,
and impose
\begin{align}
\boldsymbol u ={}& \boldsymbol g\onD \label{eq:stokes-dirichlet-boundary}\\
\bm{\sigma}(\boldsymbol u,p)\cdot\bfn ={}& \boldsymbol t\onN \label{eq:stokes-neumann-boundary}\\
\boldsymbol u\cdot\bfn ={}& 0,\qquad
(\bfI-\bfn\otimes\bfn)\bm{\sigma}(\boldsymbol u,p)\bfn=\bm{0}
\quad\text{on }\Gamma\suprm{S} \label{eq:stokes-symmetry-boundary}\\
(\bfI-\bfn\otimes\bfn)\boldsymbol u ={}& \bm{0},\qquad
\bfn\cdot\bm{\sigma}(\boldsymbol u,p)\bfn=-p\subrm{b}
\quad\text{on }\Gamma\suprm{E} \label{eq:stokes-end-boundary}
\end{align}
where $\bfn$ is the outward unit normal, $\boldsymbol t$ a prescribed
traction and $p\subrm{b}\geq 0$ a prescribed ambient pressure. On
$\Gamma\suprm{S}$ the normal velocity is essential and the tangential
traction natural; on $\Gamma\suprm{E}$ the roles are exchanged. We assume
throughout that $\Gamma\suprm{D}$ has positive measure. In the channel
tests $\Gamma\suprm{D}$ contains the walls and the prescribed inflow, and
$\Gamma\suprm{N}$ is the outlet with $\boldsymbol t = \boldsymbol 0$. In
the pit tests only wall velocities are prescribed, and the two ends belong
either to $\Gamma\suprm{N}$, with $\boldsymbol t=\bm{0}$, or to
$\Gamma\suprm{E}$. The part $\Gamma\suprm{S}$ is used only in the
three-dimensional extrusion of Section~\ref{subsec:stokes-three-dimensional}.
We use the same identification of ambient and cavitation reference
pressures as in Section~\ref{sec:reynolds-cavitation}, and impose
$p\geq 0$ in $\Omega$. The constrained problem can be written as a
variational inequality. This is the model introduced in
\cite{NiHa11} and subsequently discretised in \cite{GiHaSu10}. Let
\begin{equation}
Q_+ := \{q\in L_2(\Omega):\ q\geq 0\}
\label{eq:nonnegative-pressure-cone}
\end{equation}
and, with the spaces $V$ and $V_{\bm{g}}$ of \eqref{eq:stokes-spaces} below, seek
$\boldsymbol u\in V_{\bm{g}}$ and $p\in Q_+$ such that
\begin{equation}
\int_{\Omega} \bm{\sigma}(\boldsymbol u,0) :\bfeps({\boldsymbol v}) \dO
-\int_{\Omega} p\,\nabla\cdot{\boldsymbol v} \dO = \int_{\Omega}{
\boldsymbol f}\cdot{\boldsymbol v} \dO
+\int_{\Gamma\suprm{N}}{\boldsymbol t}\cdot{\boldsymbol v} \,\mathrm{d}s
-\int_{\Gamma\suprm{E}}p\subrm{b}\,{\boldsymbol v}\cdot\bfn \,\mathrm{d}s
\label{eq:stokes-variational-equilibrium}
\end{equation}
for all ${\boldsymbol v}\in V$, and
\begin{equation}
-\int_{\Omega}\nabla\cdot\boldsymbol u\, (q-p) \dO \leq 0,\quad
\forall q\in Q_+
\label{eq:stokes-pressure-inequality}
\end{equation}
A stabilised equal-order discretisation of this variational inequality,
with a posteriori error estimates and adaptive mesh refinement, is given
in \cite{GiHaSu10}; here we instead use augmentation.
To write the problem as a variational equality, we use the Kuhn--Tucker conditions
\begin{equation}\label{eq:stokes-complementarity}
p\geq 0, \quad \nabla\cdot\bm{u} \geq 0, \quad p\,\nabla\cdot\bm{u} = 0
\end{equation}
and
again replace conditions \eqref{eq:stokes-complementarity} by the equivalent statement
\begin{equation}\label{eq:stokes-complementarity-map}
p = \frac{1}{\gamma}[\gamma p -\nabla\cdot\bm{u}]_+
\end{equation}
with $\gamma$ a positive number.

\begin{rem}\label{rem:outflow-flux}
The parts $\Gamma\suprm{N}$ and $\Gamma\suprm{E}$, on which the normal
velocity is free, are not merely a convenience. Since
$\nabla\cdot\bm{u}\geq 0$ almost everywhere by \eqref{eq:stokes-complementarity}, the divergence
theorem gives
\begin{equation}\label{eq:outflow-flux-balance}
0\leq \int_\Omega \nabla\cdot\bm{u}\dO = \int_{\Gamma\suprm{D}}\bm{g}\cdot\bfn\,\mathrm{d}s + \int_{\Gamma\suprm{N}\cup\Gamma\suprm{E}}\bm{u}\cdot\bfn\,\mathrm{d}s
\end{equation}
so the volumetric expansion of the lubricant, integrated over the film,
equals the net outward flux through the boundary. If $\Gamma\suprm{N}\cup\Gamma\suprm{E}=\emptyset$ and the Dirichlet data
satisfy the usual compatibility condition $\int_{\partial\Omega}\bm{g}\cdot\bfn\,\mathrm{d}s = 0$,
then \eqref{eq:outflow-flux-balance} forces $\nabla\cdot\bm{u} = 0$ almost everywhere, the constraint
is saturated everywhere, and the model becomes equality-constrained
incompressible Stokes flow. The pressure is then an equality multiplier,
determined only modulo constants unless it is normalised. Thus an outflow
boundary through which the lubricant can expand is a prerequisite for
nontrivial cavitation in the present inequality formulation, for strict
feasibility, and for taking the unnormalised pressure space in
\eqref{eq:stokes-spaces} to be all of $L_2(\Omega)$.
\end{rem}

Defining function spaces
\begin{equation}\label{eq:stokes-spaces}
V_{\bm{w}} =\{\bm{v}\in [H^1(\Omega)]^n: \; \bm{v} =\bm{w}\;\text{on $\Gamma\suprm{D}$},\;
\bm{v}\cdot\bfn=0\;\text{on $\Gamma\suprm{S}$},\;
(\bfI-\bfn\otimes\bfn)\bm{v}=\bm{0}\;\text{on $\Gamma\suprm{E}$}\},
\quad V := V_{\bm{0}}, \quad Q = L_2(\Omega)
\end{equation}
and seeking $(\bm{u},p)\in V_{\bm{g}}\times Q$ equilibrium requires, with
\begin{align}
a(\bm{u},\bm{v}) &:= \int_{\Omega}\left(2\mu\,\bfeps(\bm{u}):\bfeps(\bm{v})
-\tfrac{2\mu}{3}(\nabla\cdot\bm{u})(\nabla\cdot\bm{v})\right)\dO
= \int_{\Omega}\bm{\sigma}(\bm{u},0):\bfeps(\bm{v})\dO
\label{eq:stokes-viscous-form}\\
L(\bm{v}) &:= \int_{\Omega} \bm{f}\cdot\bm{v}\dO
+ \int_{\Gamma\suprm{N}} \bm{t}\cdot\bm{v}\,\mathrm{d}s
- \int_{\Gamma\suprm{E}} p\subrm{b}\,\bm{v}\cdot\bfn\,\mathrm{d}s
\label{eq:stokes-load}
\end{align}
that
\begin{equation}
a(\bm{u},\bm{v}) -\int_{\Omega} p\, \nabla\cdot\bm{v}\dO = L(\bm{v})
\label{eq:stokes-weak-equilibrium}
\end{equation}
where $\bm{v}\in V$. Note that the traction term in $L$ and the pressure term in
$\bm{\sigma}$ enter together through \eqref{eq:stokes-neumann-boundary}, so that the natural boundary
condition delivered by this weak form is precisely \eqref{eq:stokes-neumann-boundary}, together
with the normal traction condition in \eqref{eq:stokes-end-boundary}, for the Cauchy
stress \eqref{eq:deviatoric-cauchy-stress}. To obtain an augmented Lagrangian formulation we follow \cite{ChHi13}
directly, rather than starting from a functional as in
Section~\ref{sec:reynolds-cavitation}, and write $\nabla\cdot\bm{v} = \nabla\cdot\bm{v} + \gamma q-\gamma q$
for an arbitrary function $q\in Q$, so that we may write
\begin{equation}
a(\bm{u},\bm{v}) +\int_{\Omega} p\, (\gamma q -\nabla\cdot\bm{v}) \dO -\int_{\Omega}\gamma p\, q \dO = L(\bm{v})
\label{eq:stokes-augmented-identity}
\end{equation}
Replacing the pressure in the first integral by the expression in \eqref{eq:stokes-complementarity-map} we finally obtain the problem of finding $(\bm{u},p)\in V_{\bm{g}}\times Q$ such that
\begin{equation}\label{eq:stokes-augmented-form}
a(\bm{u},\bm{v}) +\int_{\Omega} \frac{1}{\gamma}[\gamma p-\nabla\cdot\bm{u} ]_+ (\gamma q -\nabla\cdot\bm{v}) \dO -\int_{\Omega}\gamma p\, q \dO = L(\bm{v})
\quad\forall (\bm{v},q)\in V\times Q \end{equation}

This problem is related to seeking stationary points to the functional
\begin{equation}\label{eq:stokes-augmented-functional}
\mathfrak{F}(\bm{u},p) := \frac12 a(\bm{u},\bm{u}) - L(\bm{u}) + \int_{\Omega}\frac{1}{2\gamma}\left[\gamma p-\nabla\cdot\bm{u}\right]_+^2\dO -\int_{\Omega}\frac{\gamma}{2} p^2\dO
\end{equation}
analogously to \eqref{eq:reynolds-augmented-functional}.

As in Section~\ref{sec:reynolds-cavitation}, the underlying problem is a
minimisation over the convex set of velocities with nonnegative divergence.

\subsection{Continuous Well-Posedness}
\label{subsec:stokes-well-posedness}

We henceforth assume that $\Omega$ is bounded, connected, and Lipschitz,
that the boundary partition is measurable and relatively open up to sets of
surface measure zero, and that $\Gamma\suprm{N}\cup\Gamma\suprm{E}$ contains a
relatively open patch. The mixed-boundary divergence range used below can then be stated
explicitly.

\begin{lem}\label{lem:continuous-divergence-range}
For the space $V$ in \eqref{eq:stokes-spaces},
\begin{equation}\label{eq:divergence-surjectivity}
\nabla\cdot : V\longrightarrow L_2(\Omega) \quad\text{is surjective}
\end{equation}
and admits a bounded right inverse.
\end{lem}

\begin{proof}
Because $\Gamma\suprm{N}\cup\Gamma\suprm{E}$ contains a relatively open patch,
one can choose $\bm{\xi}\in V$ with boundary trace supported there, normal
to the boundary on $\Gamma\suprm{E}$, and normalise it so
that $\int_\Omega\nabla\cdot\bm{\xi}\dO=1$. Given $q\in L_2(\Omega)$,
the function
$q_0:=q-(\int_\Omega q\dO)\nabla\cdot\bm{\xi}$ has zero mean. The
standard right inverse of the divergence on $[H^1_0(\Omega)]^n$ provides
$\bm{z}$ with $\nabla\cdot\bm{z}=q_0$; see, for example,
\cite{BrFo91}. Since $[H^1_0(\Omega)]^n\subset V$, the field
$\bm{v}:=\bm{z}+(\int_\Omega q\dO)\bm{\xi}$ belongs to $V$ and has
divergence $q$. The bounded lifting and right-inverse estimates give the
boundedness assertion.
\end{proof}

\begin{thm}\label{thm:stokes-well-posedness}
Let the assumptions above hold, let $\vert\Gamma\suprm{D}\vert>0$, and
assume that the only conformal Killing field satisfying the homogeneous
essential conditions in \eqref{eq:stokes-spaces} is zero. A sufficient
condition used here is that $\Gamma\suprm{D}$ contain a relatively open
planar patch. Let $\mu>0$ be constant and let
$\mathcal{K} := \{\bm{v}\in V_{\bm{g}}:\ \nabla\cdot\bm{v}\geq 0
\text{ a.e.\ in }\Omega\}$, which is nonempty by the construction in
Remark~\ref{rem:strictly-feasible-velocity} below. Then the problem
\begin{equation}\label{eq:stokes-constrained-minimization}
\bm{u} = \arg\min_{\bm{v}\in\mathcal{K}}
\left\{\tfrac12a(\bm{v},\bm{v})-L(\bm{v})\right\}
\end{equation}
has exactly one solution, there is exactly one $p\in Q$ with
$p\geq 0$ such that $(\bm{u},p)$ satisfies \eqref{eq:stokes-complementarity} together with
$a(\bm{u},\bm{v})-(p,\nabla\cdot\bm{v})_\Omega = L(\bm{v})$ for all
$\bm{v}\in V$, and this pair solves \eqref{eq:stokes-augmented-form} for every $\gamma>0$.
Conversely, every solution of \eqref{eq:stokes-augmented-form} is this pair, so that, as in
Theorem~\ref{thm:reynolds-well-posedness}, the solution set of \eqref{eq:stokes-augmented-form} does not depend on
$\gamma$.
\end{thm}

\begin{proof}
$\mathcal{K}$ is closed and convex, and $a(\cdot,\cdot)$ is coercive on
$V$. For $n=2$ this is elementary: since
$(\nabla\cdot\bm{v})^2\leq 2\vert\bfeps(\bm{v})\vert^2$ pointwise,
\begin{equation}\label{eq:stokes-deviatoric-coercivity}
a(\bm{v},\bm{v}) = 2\mu\|\bfeps(\bm{v})\|^2_\Omega
-\tfrac{2\mu}{3}\|\nabla\cdot\bm{v}\|^2_\Omega
\geq \tfrac{2\mu}{3}\|\bfeps(\bm{v})\|^2_\Omega
\end{equation}
and we can apply Korn's inequality, available because the full velocity trace vanishes
on $\Gamma\suprm{D}$ of positive measure.

For $n=3$ the form is $2\mu\|\bfeps^{\rm dev}(\bm{v})\|_\Omega^2$, and
coercivity on $V$ follows from the trace-free Korn inequality of Dain
\cite{Da06} and the kernel-elimination hypothesis, by the compactness
argument of Peetre and Tartar; Lemma~\ref{lem:continuous-trace-free-korn}
gives the version used here. The functional in \eqref{eq:stokes-constrained-minimization} is then strictly
convex and coercive on $V_{\bm{g}}$, which gives a unique $\bm{u}$.

For the multiplier we need a constraint qualification. Slater's condition
is unavailable, since the cone of nonpositive functions has empty interior
in $L_2(\Omega)$. Robinson's condition holds instead: writing
$\bm{v} = \bm{w}_0+\bm{z}$ with $\bm{w}_0\in V_{\bm{g}}$ fixed and
$\bm{z}\in V$, \eqref{eq:divergence-surjectivity} gives
$\{-\nabla\cdot\bm{z}-\nabla\cdot\bm{w}_0 :\bm{z}\in V\} = L_2(\Omega)$, so
$0\in{\rm int}\{-\nabla\cdot\bm{z}-\nabla\cdot\bm{w}_0+c:\ \bm{z}\in V,\ c\geq 0\}$. Convex duality in the form of the Zowe--Kurcyusz theorem \cite{ZoKu79}
then provides $p$ in the dual of $L_2(\Omega)$, which is $L_2(\Omega)$
itself, lying in the dual of the nonnegative cone, which is again the
nonnegative cone, so that $p\geq 0$; and it provides
$(p,\nabla\cdot\bm{u})_\Omega=0$ together with
$a(\bm{u},\bm{v})-(p,\nabla\cdot\bm{v})_\Omega = L(\bm{v})$ for
$\bm{v}\in V$. These are \eqref{eq:stokes-complementarity}. Uniqueness of $p$ is
\eqref{eq:divergence-surjectivity} once more. Lemma
\ref{lem:complementarity-map} applied pointwise to \eqref{eq:stokes-complementarity} gives
\eqref{eq:stokes-complementarity-map}, and substituting it in the equilibrium equation returns
\eqref{eq:stokes-augmented-form}; no step involves $\gamma$.

For the converse, taking $\bm{v}=\bm{0}$ in \eqref{eq:stokes-augmented-form}
shows that $\gamma^{-1}[\gamma p-\nabla\cdot\bm{u}]_+=p$ almost everywhere,
so \eqref{eq:stokes-complementarity} holds by
Lemma~\ref{lem:complementarity-map}, and taking $q=0$ returns the
equilibrium equation. For $\bm{v}\in\mathcal{K}$,
$a(\bm{u},\bm{v}-\bm{u})-L(\bm{v}-\bm{u})=(p,\nabla\cdot\bm{v})_\Omega\geq 0$,
so $\bm{u}$ solves \eqref{eq:stokes-constrained-minimization}, and uniqueness
identifies the pair.
\end{proof}

\begin{rem}\label{rem:strictly-feasible-velocity}
A strictly feasible field exists. Let $\tilde{\bm{w}}\in[H^1(\Omega)]^n$
extend $\bm{g}$, and choose $\bm{z}_1,\bm{z}_2\in V$ with
$\nabla\cdot\bm{z}_1 = -\nabla\cdot\tilde{\bm{w}}$ and
$\nabla\cdot\bm{z}_2 = c_0$ for a constant $c_0>0$, by
\eqref{eq:divergence-surjectivity}. Then
$\bm{w} := \tilde{\bm{w}}+\bm{z}_1+\bm{z}_2\in V_{\bm{g}}$ has
$\nabla\cdot\bm{w} = c_0$. If $\Gamma\suprm{N}\cup\Gamma\suprm{E}=\emptyset$,
the range of the divergence on $V$ consists of the functions with vanishing
mean, which never contains $c_0$, so \eqref{eq:divergence-surjectivity} and
Robinson's condition fail; this is the analytical counterpart of
Remark~\ref{rem:outflow-flux}. The problem is then incompressible Stokes
flow, with the pressure in a normalised or quotient space. Without an
essential boundary condition that removes the rigid or conformal Killing
kernel, coercivity and uniqueness of $\bm{u}$ are lost as well.
\end{rem}

\subsection{Stabilised Crouzeix--Raviart Discretisation}
\label{subsec:stokes-discretization}

For the discrete problem we take the Crouzeix--Raviart velocity with
piecewise constant pressure. We assume that the mesh resolves the boundary
partition and that $\Gamma\suprm{N}\cup\Gamma\suprm{E}$ contains at least one
complete boundary face. With $\mathcal{F}^h_{\rm i}$ denoting the interior faces, set
\begin{equation}\label{eq:cr-velocity-space}
\vec{V}^h_{\bm{w}} = \bm{v}\in[L_2(\Omega)]^n:\;
\left\{\begin{aligned}
\bm{v}\vert_T&\in[\mathbb{P}_1(T)]^n &&\forall T\in\mathcal{T}^h \\
\textstyle\int_F\jump{\bm{v}}&=\bm{0} &&\forall F\in\mathcal{F}^h_{\rm i} \\
\textstyle\int_F(\bm{v}-\bm{w})&=\bm{0} &&\forall F\subset\Gamma\suprm{D} \\
\textstyle\int_F\bm{v}\cdot\bfn&=0 &&\forall F\subset\Gamma\suprm{S} \\
\textstyle\int_F(\bfI-\bfn\otimes\bfn)\bm{v}&=\bm{0} &&\forall F\subset\Gamma\suprm{E}
\end{aligned}\right.
\end{equation}
with $\vec{V}^h := \vec{V}^h_{\bm{0}}$, the jump being taken as the trace on
a boundary face, and
\begin{equation}\label{eq:piecewise-constant-pressure}
  Q^{h} = \{q\in L_2(\Omega):\; q\vert_{T}
\in \mathbb{P}_{0}(T),\,\forall T\in {\mathcal T}^h\}
\end{equation}
The degrees of freedom of $\vec{V}^h$ are the mean values over the faces $F$
of the mesh, and the essential data enter through their face means. On a
symmetry face only the normal component of the mean is constrained, and the
tangential components remain test degrees of freedom and impose the natural
zero tangential traction in \eqref{eq:stokes-symmetry-boundary}; on a face in
$\Gamma\suprm{E}$ the roles are exchanged. The full
Dirichlet patch still eliminates the global conformal kernel, so the
coercivity argument below is unchanged.

\begin{lem}\label{lem:cr-divergence-range}
The mixed-boundary Crouzeix--Raviart space satisfies
\begin{equation}\label{eq:discrete-divergence-surjectivity}
\nabla\cdot\vec{V}^h = Q^h
\end{equation}
\end{lem}

\begin{proof}
The inclusion from left to right is immediate. Let
$Q^h_0:=\{q_h\in Q^h:(q_h,1)_\Omega=0\}$. The standard
Crouzeix--Raviart/$\mathbb{P}_0$ inf--sup result gives
$\nabla\cdot\vec{V}^h_0=Q^h_0$ for the subspace
$\vec{V}^h_0$ whose face means vanish on the entire boundary. To recover
the constant mode, choose a face $F_N\subset\Gamma\suprm{N}\cup\Gamma\suprm{E}$ and let
$\psi_{F_N}$ be its scalar Crouzeix--Raviart basis function, normalised to
have unit face mean. Then
$\bm{v}_N:=\psi_{F_N}\bfn_{F_N}$ belongs to $\vec{V}^h$ and the broken
divergence theorem gives
$\int_\Omega\nabla\cdot\bm{v}_N\dO=\vert F_N\vert\neq0$.
For any $q_h\in Q^h$, scale $\bm{v}_N$ so that its divergence has the same
mean as $q_h$; the difference belongs to $Q^h_0$ and is the divergence of
a field in $\vec{V}^h_0$. Adding the two fields proves
\eqref{eq:discrete-divergence-surjectivity}.
\end{proof}

Elementwise constancy of the multiplier residual gives exact
complementarity, Theorem~\ref{thm:stokes-exact-complementarity}, and the
surjectivity of Lemma~\ref{lem:cr-divergence-range} gives feasibility and
nondegenerate pressure coupling. Exact-divergence conforming pairs such as
Scott--Vogelius pairs \cite{ScVo85} are also divergence compatible, but for
nonconstant pressure spaces the positive-part residual need not belong to
the test space.

The space $\vec{V}^h$ is not contained in $V$, and the strain-based form is
not stable on it, the discrete Korn inequality failing for nonconforming
piecewise linears \cite{Fa91}. We therefore use the stabilised form of
\cite{HaLa03},
\begin{equation}\label{eq:cr-stabilized-form}
a_h(\bm{u},\bm{v}) := \sum_{T\in\mathcal{T}^h}
\left(\bm{\sigma}(\bm{u},0),\bfeps(\bm{v})\right)_T
+ 2\mu\gamma_1\sum_{F}\left(h_F^{-1}\jump{\bm{u}},\jump{\bm{v}}\right)_F ,
\qquad h_F := \frac{\vert T^+\vert+\vert T^-\vert}{2\vert F\vert}
\end{equation}
the sum running over the interior faces and those on $\Gamma\suprm{D}$ and
$\Gamma\suprm{E}$, the jump being the trace on a boundary face, replaced on
$\Gamma\suprm{E}$ by its tangential part $(\bfI-\bfn\otimes\bfn)\bm{v}$. This
is the discontinuous Galerkin method of \cite{HaLa03}
restricted to $\vec{V}^h$: the Nitsche consistency terms drop out
identically, since $\bm{\sigma}(\bm{v},0)$ is elementwise constant for
piecewise linears while the Crouzeix--Raviart jump has vanishing mean on
every face. The penalty
restores control, and no threshold on
$\gamma_1$ is involved, in contrast with methods of Nitsche type: the
consistency terms that would carry a negative sign have dropped out, so
the penalty only adds. In two dimensions the elementwise term is bounded
below by $\tfrac{2\mu}{3}\|\bfeps(\bm{v})\|_T^2$, by the argument of
\eqref{eq:stokes-deviatoric-coercivity} applied elementwise, and the piecewise Korn inequality of
Brenner \cite{Br04} then gives, for every $\gamma_1>0$,
\begin{equation}\label{eq:cr-coercivity}
a_h(\bm{v},\bm{v})\ \geq\ c(\gamma_1)\,\|\bm{v}\|_h^2,\qquad
\|\bm{v}\|_h^2 := \sum_{T}\|\bfeps(\bm{v})\|_T^2
+\sum_{F}h_F^{-1}\|\jump{\bm{v}}\|_F^2
\end{equation}
with $c(\gamma_1)>0$ independent of $h$ but degenerating as
$\gamma_1\to 0$, where the elementwise kernel becomes nontrivial by
\cite{Fa91}; the degeneration is visible in the computations of Section
\ref{subsec:constitutive-discretization}. (The analysis of \cite{HaLa03} takes the first Lam\'e
parameter nonnegative; the elementwise bound is what carries it to
$-2\mu/3$.) In three dimensions the elementwise term is
$2\mu\Vert\bfeps^{\rm dev}(\bm{v})\Vert_T^2$, whose elementwise kernel is
the larger space of conformal Killing fields. The general
three-dimensional discrete trace-free Korn framework is that of Williams
and Hong \cite{WiHo24}; Appendix~\ref{sec:broken-korn-appendix} gives a
short proof of the piecewise linear case needed here, assuming
face-connected vertex patches and a relatively open planar patch in
$\Gamma\suprm{D}$. Consequently, \eqref{eq:cr-coercivity} holds, for every $\gamma_1>0$, in
three dimensions as well, and the
constitutive law \eqref{eq:deviatoric-cauchy-stress} is available on this pair. The penalty is
weakly consistent, vanishing on $V$.

\subsection{Exact Complementarity and Discrete Solvability}
\label{subsec:stokes-discrete-solvability}

The finite element method based on \eqref{eq:stokes-augmented-form} is to find
$(\bm{u}^h,p^h)\in \vec{V}^h_{\bm{g}}\times Q^{h}$ such that
\begin{equation}\label{eq:stokes-discrete-momentum}
a_h(\bm{u}^h,\bm{v}) -\int_{\Omega}\frac{1}{\gamma}[\gamma p^h-\nabla\cdot\bm{u}^h ]_+ \nabla\cdot\bm{v}\dO = L(\bm{v}) \quad \forall \bm{v}\in\vec{V}^h
\end{equation}
and
\begin{equation}\label{eq:stokes-discrete-multiplier}
\int_{\Omega}\left(\frac{1}{\gamma}[\gamma p^h-\nabla\cdot\bm{u}^h ]_+- p^h\right)q\dO =0  , \quad \forall q\in Q^{h}
\end{equation}
Because $\nabla\cdot\bm{u}^h$ and $p^h$ are both elementwise constant,
\eqref{eq:stokes-discrete-multiplier} is
the pointwise
identity $p^h = \gamma^{-1}[\gamma p^h-\nabla\cdot\bm{u}^h]_+$ on every
element.

Freezing the sign of $\gamma p^h-\nabla\cdot\bm{u}^h$ at the quadrature
points, as in Section~\ref{sec:reynolds-cavitation}, and writing $\mathcal{A}$ for the
subset of $\Omega$ on which the positive part is active, with ties assigned
to $\mathcal{A}$, the resulting
linear problem is: find $(\bm{u}^h,p^h)$ with
\begin{align}
a_h(\bm{u}^h,\bm{v}) + \gamma^{-1}(\nabla\cdot\bm{u}^h,\nabla\cdot\bm{v})_{\mathcal{A}}
- (p^h,\nabla\cdot\bm{v})_{\mathcal{A}} ={}& L(\bm{v}) \label{eq:stokes-frozen-momentum}\\
-(\nabla\cdot\bm{u}^h,q)_{\mathcal{A}} - \gamma(p^h,q)_{\Omega\setminus\mathcal{A}} ={}& 0
\label{eq:stokes-frozen-multiplier}
\end{align}
for all $(\bm{v},q)\in\vec{V}^h\times Q^h$. Unlike its Reynolds
counterpart \eqref{eq:reynolds-nodal-system} this is a genuine saddle point problem, and
its solvability rests on precisely the nondegeneracy of the
pressure--velocity coupling of the underlying Stokes pair.

The following result is the elementwise counterpart of
Theorem~\ref{thm:reynolds-nodal-structure}; it uses elementwise closure of
the multiplier residual rather than surjectivity of the divergence.

\begin{thm}\label{thm:stokes-exact-complementarity}
Let $\mathcal{A}$ be any subset of the elements and let $(\bm{u}^h,p^h)$
be any solution of \eqref{eq:stokes-frozen-momentum}--\eqref{eq:stokes-frozen-multiplier}. Then
\begin{equation}
(\nabla\cdot\bm{u}^h)_T = 0 \ \ \text{for }T\in\mathcal{A},
\qquad
p^h_T = 0 \ \ \text{for }T\notin\mathcal{A}
\label{eq:stokes-active-set-structure}
\end{equation}
so that the product $p^h\,\nabla\cdot\bm{u}^h$ vanishes identically, on
every element and for every $\mathcal{A}$. The test argument
$\gamma p^h_T-(\nabla\cdot\bm{u}^h)_T$ reduces to $\gamma p^h_T$ on
$\mathcal{A}$ and to $-(\nabla\cdot\bm{u}^h)_T$ off it. At a fixed point,
membership therefore gives $p^h_T\geq0$ on $\mathcal{A}$ and
$(\nabla\cdot\bm{u}^h)_T>0$ off it, because the test fails there. Neither
the frozen solutions nor the fixed point depend on $\gamma$, and the fixed
point satisfies the discrete Kuhn--Tucker
conditions $p^h\geq 0$, $\nabla\cdot\bm{u}^h\geq 0$,
$p^h\,\nabla\cdot\bm{u}^h = 0$ exactly, elementwise; in particular its
pressure lies in the continuous constraint set $Q_+$.
\end{thm}

\begin{proof}
Since $\nabla\cdot\bm{u}^h$ and $q$ are elementwise constant, \eqref{eq:stokes-frozen-multiplier}
tested with the characteristic function of a single element $T$ reads
$-(\nabla\cdot\bm{u}^h)_T\vert T\vert = 0$ for $T\in\mathcal{A}$ and
$-\gamma p^h_T\vert T\vert = 0$ otherwise, which is the first assertion, and
on each element one factor of the product therefore vanishes. The
reduction of the test argument follows at once, and $\gamma$ enters
neither reduced form. A fixed point is a solution for which the
membership test is satisfied on every element: $p^h_T\geq 0$ on
$\mathcal{A}$, where $\nabla\cdot\bm{u}^h$ vanishes, and, off
$\mathcal{A}$, failure of the test, which with $p^h_T = 0$ there reads
$(\nabla\cdot\bm{u}^h)_T>0$.
\end{proof}

\begin{rem}\label{rem:continuous-pressure-contrast}
With a continuous pressure space
\eqref{eq:stokes-discrete-multiplier} says only that
$\gamma^{-1}[\gamma p^h-\nabla\cdot\bm{u}^h]_+-p^h$ is $L_2$ orthogonal to
$Q^h$, not that it vanishes; the discrete pressure need not be nonnegative,
and residual violations can survive. Here elementwise constancy of the
multiplier residual turns orthogonality into a pointwise identity, Lemma~\ref{lem:complementarity-map} applies
elementwise, and the violation is not merely small but absent. In a Taylor--Hood computation of the channel
problem of Section~\ref{sec:numerical-results} the violation measured $-5.8\cdot 10^{-3}$,
$-6.3\cdot 10^{-4}$, $-6.4\cdot 10^{-5}$ and $-6.4\cdot 10^{-6}$ for
$\gamma = 10, 10^2, 10^3, 10^4$, successive ratios $9.2$, $9.9$ and
$10.0$. The Stokes
model thereby acquires the same structure the Reynolds model has in Section
\ref{subsec:reynolds-nodal-method}: complementarity conditions that hold exactly, and a
method from which the augmentation parameter has disappeared.
\end{rem}

\begin{thm}\label{thm:stokes-discrete-solvability}
Suppose $\vert\Gamma\suprm{D}\vert>0$, let $a_h(\cdot,\cdot)$ be positive
definite on $\vec{V}^h$, and let the pair
$(\vec{V}^h,Q^h)$ be nondegenerate, in the sense that
$(q,\nabla\cdot\bm{v})_\Omega = 0$ for all $\bm{v}\in\vec{V}^h$ implies
$q=0$ for $q\in Q^h$. Then \eqref{eq:stokes-frozen-momentum}--\eqref{eq:stokes-frozen-multiplier} has exactly one
solution, for every $\gamma>0$ and every $\mathcal{A}$. For
\eqref{eq:cr-stabilized-form} the first hypothesis holds for every
$\gamma_1>0$, in two dimensions by \eqref{eq:cr-coercivity} and in three by
Corollary~\ref{cor:three-dimensional-coercivity}; it holds for the full-gradient form as well, since an elementwise constant field with continuous
face means vanishing on $\Gamma\suprm{D}$ is zero. For the
Crouzeix--Raviart pair the second hypothesis follows from
\eqref{eq:discrete-divergence-surjectivity}.
\end{thm}

\begin{proof}
The problem is square, so it suffices to treat $L=0$ and show that
$(\bm{u}^h,p^h)=(\bm{0},0)$. Take $\bm{v}=\bm{u}^h$ in \eqref{eq:stokes-frozen-momentum} and
$q=p^h$ in \eqref{eq:stokes-frozen-multiplier}. The second gives
$(\nabla\cdot\bm{u}^h,p^h)_{\mathcal{A}}
= -\gamma\|p^h\|^2_{\Omega\setminus\mathcal{A}}$, and substituting this in
the first,
\begin{equation}
a_h(\bm{u}^h,\bm{u}^h)
+ \gamma^{-1}\|\nabla\cdot\bm{u}^h\|^2_{\mathcal{A}}
+ \gamma\|p^h\|^2_{\Omega\setminus\mathcal{A}} = 0
\label{eq:stokes-frozen-energy}
\end{equation}
All three terms are nonnegative, so each vanishes. From the first and the
definiteness of $a_h(\cdot,\cdot)$, $\bm{u}^h=\bm{0}$; from the third, $p^h=0$ on
$\Omega\setminus\mathcal{A}$. Equation \eqref{eq:stokes-frozen-momentum} now reads
$(p^h,\nabla\cdot\bm{v})_{\mathcal{A}}=0$ for all $\bm{v}\in\vec{V}^h$,
and since $p^h$ vanishes off $\mathcal{A}$ the restricted pairing is the
full one, so $(p^h,\nabla\cdot\bm{v})_\Omega=0$ for all
$\bm{v}\in\vec{V}^h$. Nondegeneracy gives $p^h=0$.
\end{proof}

\begin{thm}\label{thm:stokes-nonlinear-discrete-solvability}
Under the hypotheses of Theorem~\ref{thm:stokes-discrete-solvability},
assume also \eqref{eq:discrete-divergence-surjectivity}. Then the nonlinear
system \eqref{eq:stokes-discrete-momentum}--\eqref{eq:stokes-discrete-multiplier}
has a unique solution for every $\gamma>0$.
\end{thm}

\begin{proof}
Choose a discrete lifting $\bm{w}^h\in\vec{V}^h_{\bm{g}}$. Surjectivity
provides $\bm{z}^h\in\vec{V}^h$ with
$\nabla\cdot\bm{z}^h=1-\nabla\cdot\bm{w}^h$, so the set
\begin{equation}\label{eq:stokes-discrete-feasible-set}
\mathcal{K}^h := \{\bm{v}\in\vec{V}^h_{\bm{g}}:
\nabla\cdot\bm{v}\geq0\}
\end{equation}
is nonempty and contains a strictly feasible point. Positive definiteness
of $a_h$ on the homogeneous space makes
$\tfrac12a_h(\bm{v},\bm{v})-L(\bm{v})$ strictly convex and coercive
on this finite-dimensional affine space. It therefore has a unique
minimiser over the closed convex set $\mathcal{K}^h$.
The finite-dimensional Kuhn--Tucker conditions give a nonnegative
$p^h\in Q^h$ satisfying equilibrium and elementwise complementarity.
If two pressures correspond to the same velocity, their difference pairs
to zero with every discrete divergence; surjectivity makes them equal.
Lemma~\ref{lem:complementarity-map} now identifies these conditions with
\eqref{eq:stokes-discrete-momentum}--\eqref{eq:stokes-discrete-multiplier}
for every $\gamma>0$. Conversely, the elementwise multiplier equation
implies precisely these Kuhn--Tucker conditions.
\end{proof}

\subsection{Consistency and Error Estimate}
\label{subsec:stokes-consistency}

\begin{thm}\label{thm:stokes-consistency}
Let $(\bm{u},p)$ be as in Theorem~\ref{thm:stokes-well-posedness}, with
$\bm{u}\in[H^2(\Omega)]^n$ and $p\in H^1(\Omega)$. Then $(\bm{u},p)$
satisfies \eqref{eq:stokes-discrete-multiplier} exactly, for all $q\in Q^h$ and every $\gamma>0$, and
satisfies \eqref{eq:stokes-discrete-momentum} up to the nonconformity term
\begin{equation}\label{eq:stokes-consistency-residual}
\begin{aligned}
E_h(\bm{v}) := {}&\sum_{F}\left(\bm{\sigma}(\bm{u},p)\cdot\bfn
-\overline{\bm{\sigma}(\bm{u},p)\cdot\bfn},\ \jump{\bm{v}}\right)_F
\\
&+\sum_{F\subset\Gamma\suprm{S}}\left(
\bfn\cdot\bm{\sigma}(\bm{u},p)\cdot\bfn
-\overline{\bfn\cdot\bm{\sigma}(\bm{u},p)\cdot\bfn},\
\ \bm{v}\cdot\bfn\right)_F
\\
&+\sum_{F\subset\Gamma\suprm{E}}\left(
(\bfI-\bfn\otimes\bfn)\bm{\sigma}(\bm{u},p)\bfn
-\overline{(\bfI-\bfn\otimes\bfn)\bm{\sigma}(\bm{u},p)\bfn},\
(\bfI-\bfn\otimes\bfn)\bm{v}\right)_F
\\
&+2\mu\gamma_1\sum_{F\subset\Gamma\suprm{D}}
\left(h_F^{-1}\left(\bm{g}-\overline{\bm{g}}\right),\jump{\bm{v}}\right)_F
\end{aligned}
\end{equation}
for $\bm{v}\in\vec{V}^h$, the first sum running over the interior faces
and those on $\Gamma\suprm{D}$, the bar denoting the mean over the face.
The second and third sums are the traction fluctuations on $\Gamma\suprm{S}$
and $\Gamma\suprm{E}$, and the last term is the penalty evaluated at the
data; it vanishes
whenever $\bm{g}$ is constant on each face of $\Gamma\suprm{D}$, as it is
in every pit computation of Section~\ref{sec:numerical-results}. Define
\begin{equation}\label{eq:stokes-boundary-data-oscillation}
G_h(\bm{g}) := \left(\sum_{F\subset\Gamma\suprm{D}}
h_F^{-1}\|\bm{g}-\overline{\bm{g}}\|_F^2\right)^{1/2}
\end{equation}
Then, in the norm of \eqref{eq:cr-coercivity},
\begin{equation}\label{eq:stokes-consistency-bound}
|E_h(\bm{v})| \leq
\left[Ch\left(|\bm{u}|_{H^2(\Omega)}+|p|_{H^1(\Omega)}\right)
+2\mu\gamma_1G_h(\bm{g})\right]\|\bm{v}\|_h
\end{equation}
where $C$ is independent of $h$. In particular the first-order bound
holds when the Dirichlet data are constant on each boundary face.
\end{thm}

\begin{proof}
By \eqref{eq:stokes-complementarity} and Lemma~\ref{lem:complementarity-map},
$\gamma^{-1}[\gamma p-\nabla\cdot\bm{u}]_+ = p$ almost everywhere, so
\eqref{eq:stokes-discrete-multiplier} becomes $(p-p,q)_\Omega = 0$, and \eqref{eq:stokes-discrete-momentum} asks whether
$a_h(\bm{u},\bm{v})-(p,\nabla\cdot\bm{v})_\Omega = L(\bm{v})$ on
$\vec{V}^h$. In the penalty of \eqref{eq:cr-stabilized-form} the interior jumps of
$\bm{u}$ vanish; on $\Gamma\suprm{D}$ the trace of $\bm{u}$ is $\bm{g}$,
and since $\int_F\jump{\bm{v}} = \bm{0}$ for $\bm{v}\in\vec{V}^h$ the
face mean may be subtracted, which is the last term of \eqref{eq:stokes-consistency-residual}.
Integrating the elementwise form by parts and
using \eqref{eq:stokes-strong} and \eqref{eq:stokes-neumann-boundary} leaves the face terms
$\sum_F(\bm{\sigma}(\bm{u},p)\cdot\bfn,\jump{\bm{v}})_F$ over the
interior and $\Gamma\suprm{D}$ faces, the contribution on
$\Gamma\suprm{N}$ cancelling against the traction term of $L$. On a symmetry
face the tangential traction vanishes, leaving
$(\bfn\cdot\bm{\sigma}\cdot\bfn,\bm{v}\cdot\bfn)_F$; the face mean of
$\bm{v}\cdot\bfn$ is zero by \eqref{eq:cr-velocity-space}, so the mean
normal traction may be subtracted. On a face in $\Gamma\suprm{E}$ the normal
traction cancels against the end term of $L$, the tangential trace of
$\bm{u}$ vanishes, and the tangential face mean of $\bm{v}$ is zero, so the
mean tangential traction may be subtracted. These observations give the
first three terms of \eqref{eq:stokes-consistency-residual}. Trace
inequalities and approximation of the traction by its face mean give the
first-order bound for the first three sums. For the Dirichlet penalty, Cauchy--Schwarz gives
\begin{equation}\label{eq:stokes-dirichlet-penalty-bound}
\left|2\mu\gamma_1\sum_{F\subset\Gamma\suprm{D}}
(h_F^{-1}(\bm{g}-\overline{\bm{g}}),\jump{\bm{v}})_F\right|
\leq 2\mu\gamma_1G_h(\bm{g})\|\bm{v}\|_h
\end{equation}
Combining these estimates proves \eqref{eq:stokes-consistency-bound}.
\end{proof}

The exact discrete complementarity also yields an error estimate. Let $I_h$
be the Crouzeix--Raviart interpolant, defined by
$\int_F I_h\bm{v} = \int_F\bm{v}$ on every face, and let $\Pi_0$ be the
$L_2$ projection onto $Q^h$. The divergence theorem on each element gives
\begin{equation}\label{eq:cr-interpolant-divergence}
\nabla\cdot(I_h\bm{v})\vert_T = (\Pi_0\nabla\cdot\bm{v})\vert_T
\qquad\forall T\in\mathcal{T}^h
\end{equation}
so the interpolant of a feasible velocity is discretely feasible.

\begin{thm}\label{thm:stokes-error}
Let the assumptions of Theorem~\ref{thm:stokes-consistency} hold and let
$(\bm{u}^h,p^h)$ be the solution of
\eqref{eq:stokes-discrete-momentum}--\eqref{eq:stokes-discrete-multiplier}.
Then, with the norm of \eqref{eq:cr-coercivity} extended to
$[H^2(\Omega)]^n+\vec{V}^h$,
\begin{equation}\label{eq:stokes-error-estimate}
\|\bm{u}-\bm{u}^h\|_h+\|p-p^h\|_\Omega\leq
C\left[h\left(|\bm{u}|_{H^2(\Omega)}+|p|_{H^1(\Omega)}\right)
+G_h(\bm{g})\right]
\end{equation}
where $C$ is independent of $h$ and $\gamma$.
\end{thm}

\begin{proof}
The face means of $\bm{u}$ satisfy the constraints in
\eqref{eq:cr-velocity-space}, so $I_h\bm{u}\in\vec{V}^h_{\bm{g}}$, and
$\|\bm{u}-I_h\bm{u}\|_h\leq Ch|\bm{u}|_{H^2(\Omega)}$. Subtracting
\eqref{eq:stokes-discrete-momentum}, in which the positive part equals
$p^h$ by \eqref{eq:stokes-discrete-multiplier}, from the identity of
Theorem~\ref{thm:stokes-consistency} gives
\begin{equation}\label{eq:stokes-error-equation}
a_h(\bm{u}-\bm{u}^h,\bm{v})-(p,\nabla_h\cdot\bm{v})_\Omega
+(p^h,\nabla\cdot\bm{v})_\Omega = E_h(\bm{v})
\qquad\forall\bm{v}\in\vec{V}^h
\end{equation}
Put $\bm{e}^h := I_h\bm{u}-\bm{u}^h$. By
Theorem~\ref{thm:stokes-exact-complementarity} and
\eqref{eq:cr-interpolant-divergence},
$(p^h,\nabla\cdot\bm{e}^h)_\Omega=(p^h,\Pi_0\nabla\cdot\bm{u})_\Omega\geq 0$.
Since $p\geq 0$, $\nabla\cdot\bm{u}^h\geq 0$ and
$(p,\nabla\cdot\bm{u})_\Omega=0$,
\begin{equation}\label{eq:stokes-error-sign}
(p,\nabla_h\cdot\bm{e}^h)_\Omega\leq(p,\Pi_0\nabla\cdot\bm{u})_\Omega
=(\Pi_0p-p,\nabla\cdot\bm{u}-\Pi_0\nabla\cdot\bm{u})_\Omega
\leq Ch^2|p|_{H^1(\Omega)}|\bm{u}|_{H^2(\Omega)}
\end{equation}
Taking $\bm{v}=\bm{e}^h$ in \eqref{eq:stokes-error-equation} and writing
$\bm{e}^h=(I_h\bm{u}-\bm{u})+(\bm{u}-\bm{u}^h)$, the coercivity
\eqref{eq:cr-coercivity}, the boundedness of $a_h(\cdot,\cdot)$ in
$\|\cdot\|_h$ and \eqref{eq:stokes-consistency-bound} give
\begin{equation}\label{eq:stokes-error-energy}
c(\gamma_1)\|\bm{e}^h\|_h^2\leq
\left[Ch\left(|\bm{u}|_{H^2(\Omega)}+|p|_{H^1(\Omega)}\right)
+2\mu\gamma_1G_h(\bm{g})\right]\|\bm{e}^h\|_h
+Ch^2|p|_{H^1(\Omega)}|\bm{u}|_{H^2(\Omega)}
\end{equation}
and the velocity bound follows by Young's inequality and the triangle
inequality. For the pressure, \eqref{eq:stokes-error-equation} reads
$(\Pi_0p-p^h,\nabla\cdot\bm{v})_\Omega
= a_h(\bm{u}-\bm{u}^h,\bm{v})-E_h(\bm{v})$ for $\bm{v}\in\vec{V}^h$. The
Crouzeix--Raviart pair satisfies a discrete inf--sup condition uniformly in
$h$ \cite{BrFo91}, with $I_h$ as Fortin operator and
Lemma~\ref{lem:continuous-divergence-range} as its continuous counterpart,
so $\|\Pi_0p-p^h\|_\Omega$ is bounded by the right-hand side of
\eqref{eq:stokes-error-estimate}. Finally
$\|p-\Pi_0p\|_\Omega\leq Ch|p|_{H^1(\Omega)}$.
\end{proof}

\begin{rem}
The two ways of failing to be conforming have been exchanged: with a
continuous pair the spaces are conforming and the constraint is met only
weakly, while here the constraint is met exactly, by
Theorem~\ref{thm:stokes-exact-complementarity}, and the velocity space is
nonconforming, with the residual bounded by
\eqref{eq:stokes-consistency-bound}. We prefer the latter because the
extent of the cavitated region, which the constraint determines, is the
quantity the computations of Section~\ref{sec:numerical-results} find
hardest to pin down, and because the exact constraint gives the short
proof of Theorem~\ref{thm:stokes-error}. That estimate says nothing
directly about the position of the free boundary, and its regularity
assumptions are global; the channel examples can be singular where
$\Gamma\suprm{D}$ meets $\Gamma\suprm{N}$. Penalising $\bm{u}-\bm{g}$ on
$\Gamma\suprm{D}$ instead of $\bm{u}$, with the corresponding data term
moved to $L$, would remove the data term of
\eqref{eq:stokes-consistency-residual}; we keep
\eqref{eq:cr-stabilized-form}, which is the form computed with. Its
boundary-data residual vanishes in every pit computation but is present
for the nonconstant channel inflow, where facewise $H^1$ data on a
quasiuniform boundary mesh give only
$G_h(\bm{g})\lesssim h^{1/2}\|\nabla_\tau\bm{g}\|_{\Gamma\suprm{D}}$.
\end{rem}

\begin{rem}\label{rem:deviatoric-stability}
For fields vanishing on the whole boundary,
$2\mu(\bfeps(\bm{u}),\bfeps(\bm{v}))_\Omega
=\mu(\nabla\bm{u},\nabla\bm{v})_\Omega+\mu(\nabla\cdot\bm{u},\nabla\cdot\bm{v})_\Omega$,
so the usual viscous forms agree on divergence-free fields and differ only
on the cavitated set, where the choice is one of bulk viscosity. With
$\vert\Gamma\suprm{N}\cup\Gamma\suprm{E}\vert>0$ the choice also changes the
natural boundary condition, as for `do nothing' outflow conditions. The full-gradient form needs only the Poincar\'e inequality for coercivity, whereas
\eqref{eq:deviatoric-cauchy-stress} needs Korn's inequality, through
\eqref{eq:stokes-deviatoric-coercivity} for $n=2$ and in the trace-free form
of Dain \cite{Da06} for $n=3$. With $\tfrac12$ in place of $\tfrac13$ in
plane flow the bound \eqref{eq:stokes-deviatoric-coercivity} would
degenerate, so the physically correct factor is also the one for which
coercivity is elementary.
\end{rem}

\begin{rem}\label{rem:structure-preserving-parallel}
The parallel with the nodal Reynolds method is exact: in both cases the
discrete multiplier equation is a pointwise statement rather than an
orthogonality, by nodal quadrature for piecewise linears there and by
elementwise constancy here. The price is a nonconforming velocity space,
with the consistency cost \eqref{eq:stokes-consistency-bound}, and a
pressure of lower degree than that of Taylor--Hood;
Section~\ref{subsec:constitutive-discretization} compares the two. The
frozen systems are insensitive to $\gamma$, in contrast with
Theorem~\ref{thm:reynolds-stability}, because the term carrying
$\gamma^{-1}$ has the same sign as $a_h(\cdot,\cdot)$.
\end{rem}


\section{Nonlinear Solution Strategies}
\label{sec:nonlinear-solvers}

\subsection{Primal--Dual Active-Set Iteration}
\label{subsec:primal-dual-active-set}

For both flow models, the positive-part operator can be treated by freezing
the sign of its argument. Given an iterate, the active region contains the
quadrature points, nodes, or elements at which that argument is nonnegative;
the corresponding linear system is solved and the set is updated. An
unchanged set reproduces the same linear system and is therefore an exact
fixed point. The systems are nonsingular for the structure-preserving
Reynolds and Stokes pairs analysed above, but this does not imply that the
set iteration terminates. For these unsmoothed active-set solves we use an
unchanged set as the stopping criterion. On degenerate problems the set
can fail to repeat although the iterates have converged: it then changes
only by elements, or nodes, at which both factors of the complementarity
product vanish to rounding error, and every frozen solution satisfies the
Kuhn--Tucker conditions to that accuracy. We then use the sign rule, which
stops as soon as a frozen solution satisfies the sign conditions to a
tolerance $\tau$; complementarity holds for every frozen solution by
Theorems~\ref{thm:reynolds-nodal-structure}
and~\ref{thm:stokes-exact-complementarity}. We take $\tau=10^{-12}$. The
smoothed and quadratic programming solves use the residual and optimality
checks specified in Section~\ref{subsec:solver-robustness}.

\subsection{Smoothed Central-Path Formulation}
\label{subsec:smoothed-central-path}

The active-set iteration has no termination guarantee, and its iteration
count can grow under refinement, as Section~\ref{subsec:reynolds-pit}
shows. We therefore also consider a differentiable complementarity map
that fits the same augmented framework.

Replace the positive part by
\begin{equation}\label{eq:smoothing-function}
\varphi_s(w) := \frac{w}{2}+\sqrt{\frac{w^2}{4}+s},\qquad s>0
\end{equation}
the smoothing function of Chen, Harker, Kanzow and Smale
\cite{ChHa93,Ka96,Sm87}, which is smooth for $s>0$ and reduces to
$[\,\cdot\,]_+$ at $s=0$. Its use in a Nitsche formulation of contact, as a logarithmic barrier whose
solutions trace the central path of an interior point method, is analysed
in \cite{HaLaLa25}. Lemma~\ref{lem:complementarity-map} has an exact
counterpart.

\begin{lem}\label{lem:smoothed-complementarity}
Let $\gamma,s>0$ and $a,b\in{\mathbb{R}}$. Then
\begin{equation}
a = \frac{1}{\gamma}\varphi_s\left(\gamma a-b\right)
\qquad\Longleftrightarrow\qquad
a>0,\quad b>0,\quad ab = \frac{s}{\gamma}
\label{eq:smoothed-complementarity-equivalence}
\end{equation}
\end{lem}

\begin{proof}
By construction $\varphi_s(w)$ is the unique positive root of
$z^2-wz-s = 0$, and
\begin{equation}\label{eq:smoothing-identities}
\varphi_s(w)-w = \varphi_s(-w),\qquad
\varphi_s(w)\,\varphi_s(-w) = s
\end{equation}
the first by inspection and the second because the product telescopes to
$(w^2/4+s)-w^2/4$. Suppose $a = \gamma^{-1}\varphi_s(w)$ with
$w=\gamma a-b$. Then $b = \gamma a-w = \varphi_s(w)-w = \varphi_s(-w)$, so
$a$ and $b$ are positive, and $\gamma ab = \varphi_s(w)\varphi_s(-w) = s$.
Conversely, let $a,b>0$ with $ab = s/\gamma$ and put $w = \gamma a-b$. Then
$(\gamma a)^2-w(\gamma a)-s = \gamma ab-s = 0$ and $\gamma a>0$, so
$\gamma a = \varphi_s(w)$.
\end{proof}

Against Lemma~\ref{lem:complementarity-map}, the complementarity condition $ab=0$ is perturbed
into $ab = s/\gamma$. For the Stokes model that reads
$p\,\nabla\cdot\bm{u} = s/\gamma$: the pressure and the divergence are both
strictly positive, and no point of the film is either exactly cavitated or
exactly intact. This is the central path, and the regularisation acts as an
interior point relaxation rather than as a penalty.

For the solver experiments we take the Crouzeix--Raviart pair with the
full-gradient viscous form and no jump penalty from
Section~\ref{subsec:constitutive-discretization}, both
because it is the case in which the set-repetition rule fails and because,
having $\nabla\cdot\vec{V}^h = Q^h$, it carries the identity of Lemma
\ref{lem:smoothed-complementarity} elementwise rather than only weakly. Write
$w_T = \gamma p_T-(\nabla\cdot\bm{u}^h)_T$ and
$\lambda = \gamma^{-1}\varphi_s(w)$. With $\mathsf{K}$ the stiffness
matrix, $\mathsf{B}$ the divergence matrix,
$(\mathsf{B}u)_T = \int_T\nabla\cdot\bm{u}^h$, and $\mathsf{M}$ the
diagonal matrix of element areas,
the method is $R(u,p)=0$ with
\begin{equation}\label{eq:smoothed-newton-residual}
R_u = \mathsf{K}u - \mathsf{B}^{\rm T}\lambda - \mathsf{F},\qquad
R_p = \gamma \mathsf{M}(\lambda-p)
\end{equation}
and, with $\mathsf{D} := {\rm diag}\,\varphi_s'(w_T)$, its Jacobian is
\begin{equation}\label{eq:smoothed-newton-jacobian}
\mathsf{J} = \left[\begin{array}{cc}
\mathsf{K}+\gamma^{-1}\mathsf{B}^{\rm T}\mathsf{D}\mathsf{M}^{-1}\mathsf{B} & -\mathsf{B}^{\rm T}\mathsf{D} \\
-\mathsf{D}\mathsf{B} & \gamma \mathsf{M}(\mathsf{D}-\mathsf{I})\end{array}\right]
\end{equation}

Since $\varphi_s'$ lies strictly between $0$ and $1$, the $(2,2)$ block of
\eqref{eq:smoothed-newton-jacobian} is negative definite. If $\mathsf{K}$ is
positive definite on the free velocity degrees of freedom, as the boundary
conditions and the stabilisation of Section~\ref{sec:stokes-cavitation}
ensure, and for the full-gradient form the discrete Poincar\'e inequality,
the $(1,1)$ block is positive definite. The Jacobian is then symmetric
quasidefinite and nonsingular for every $s>0$ and every iterate, without the
nondegeneracy hypothesis of Theorem~\ref{thm:stokes-discrete-solvability}.
As $s\to 0$, away from switching points, $\mathsf{D}$ tends to the indicator
of $\mathcal{A}$ and the frozen system
\eqref{eq:stokes-frozen-momentum}--\eqref{eq:stokes-frozen-multiplier} is
recovered; at a zero switching argument the derivative tends to $1/2$
instead.


\section{Numerical Verification and Model Comparison}
\label{sec:numerical-results}

In this section we compare the nodal mixed and multiplier-free stabilised
Reynolds methods, test the jump-stabilised deviatoric Stokes method in two
and three dimensions, and compare the Reynolds and Stokes models on a pitted
surface. The Stokes jump parameter is \(\gamma_1=1\), except in the stated
penalty sweep. Taylor--Hood and full-gradient CR comparisons in
Section~\ref{subsec:constitutive-discretization} distinguish changes in
approximation space, constitutive law, and natural traction.

The unsmoothed experiments use the primal--dual active-set iteration
\cite{HiItKu02}, stopped when the set repeats or, where stated, by the
sign rule of Section~\ref{subsec:primal-dual-active-set}. The model
comparison of Section~\ref{subsec:model-comparison} and the reference
solutions of Section~\ref{subsec:solver-robustness} use quadratic
programming with working-set refinement. For every reported solution we
check stationarity, the signs of pressure and divergence, and
complementarity.

The accompanying MATLAB directory contains the source and drivers used for
these experiments.
No random numbers are used. The reported datasets were obtained with MATLAB R2025a and R2025b.

\subsection{Reynolds Pit Problem}
\label{subsec:reynolds-pit}

Let $\Omega = (0,3)\times(0,1)$ and let the workpiece carry one pit,
\begin{equation}\label{eq:reynolds-pit-geometry}
d(x,y) = 1 + \delta\, {\rm e}^{-\varrho},\qquad
\varrho = \frac{(x-x_{\rm p})^2+(y-y_{\rm p})^2}{r^2}
\end{equation}
with $\delta = 1$, $r = 0.35$ and $(x_{\rm p},y_{\rm p}) = (1.5,0.5)$, so
that the pit is as deep as the nominal film is thick. Since $f = -\partial
d/\partial x$, the load is negative over the upstream half of the pit,
where the gap diverges, and positive over the downstream half, where it
converges. The upstream half is therefore the one that cavitates.

We compute first with the conforming mixed discretisation of Section
\ref{subsec:reynolds-nodal-method}. Table~\ref{tab:reynolds-nodal-refinement} shows the behaviour under uniform refinement. At
refinement level 5, the active nodes project onto
$x\in[0.03125,1.28125]$, and the discrete pressure attains its maximum at
$x=1.9375$, just downstream of the pit; see Figure~\ref{fig:reynolds-pit-pressure}, and
Figure~\ref{fig:reynolds-pit-multiplier} for the multiplier.

\begin{table}[htbp]
\centering
\begin{tabular}{rrrrrr}
\toprule
refinement & nodes & iterations & $\max P_h$ & cavitated fraction & $\min P_h$ \\
\midrule
2 & 65 & 3 & 0.03999 & 0.231 & $-3.5\cdot 10^{-33}$ \\
3 & 225 & 4 & 0.04079 & 0.311 & $-1.7\cdot 10^{-35}$ \\
4 & 833 & 8 & 0.04181 & 0.353 & $-9.2\cdot 10^{-36}$ \\
5 & 3201 & 13 & 0.04193 & 0.383 & $-5.4\cdot 10^{-34}$ \\
6 & 12545 & 22 & 0.04194 & 0.398 & $-2.1\cdot 10^{-33}$ \\
\bottomrule
\end{tabular}
\caption{Reynolds model, conforming mixed discretisation of
Section~\ref{subsec:reynolds-nodal-method}. Uniform refinement of the pit
problem \eqref{eq:reynolds-pit-geometry}. The cavitated fraction is the
fraction of nodes in $\mathcal{A}$.}
\label{tab:reynolds-nodal-refinement}
\end{table}

First, at refinement level 4, complementarity holds to round-off: $P_i = 0$
exactly on $\mathcal{A}$ and $\lambda_i = 0$ off it, so
$\max_i\vert\lambda_iP_i\vert = 1.7\cdot 10^{-34}$ and the discrete
Kuhn--Tucker residual $\|AP-F-m\lambda\|_\infty$ is $1.2\cdot 10^{-16}$
against $\|F\|_\infty = 9.4\cdot 10^{-3}$. Second, and as a consequence,
$\gamma$ cancels from the method altogether: the test
$\gamma\lambda_i-P_i$ reduces to $\gamma\lambda_i$ on $\mathcal{A}$
and to $-P_i$ off it. At a fixed point, membership gives
$\lambda_i\geq0$ on $\mathcal{A}$ and $P_i>0$ off it. Varying $\gamma$
over eight orders of
magnitude, from $10^{-4}$ to $10^{4}$, left both the computed pressure and
the iteration count unchanged, the former to $5\cdot 10^{-16}$ in
relative $\ell_2$ norm. The method is thus exactly the classical
primal--dual active set method, and $\gamma$ survives only as a formality
inherited from the continuous formulation.

Third, the iteration count is \emph{not} mesh independent. The active set
is initialised from the unconstrained solution, which overestimates the
cavitated region, and it then recedes by one layer of nodes per iteration:
a node leaves $\mathcal{A}$ only when its multiplier turns negative, which
happens only at the edge of the set. The cost remains modest on the meshes in
Table~\ref{tab:reynolds-nodal-refinement}, but it grows.

Fourth, the cavitated fraction in Table~\ref{tab:reynolds-nodal-refinement}
has not settled. Upstream of the pit the load decays like the Gaussian, so
towards the inflow boundary the multiplier lies many orders of magnitude
below the discretisation error and the classification of those nodes is not
resolved. The fraction is therefore an ill-conditioned diagnostic, whereas
the peak pressure changes only in the fifth digit over the last
refinement.

As an independent check the same discrete problem, which is the quadratic
program of minimising $\frac12 P^{\rm T}AP - F^{\rm T}P$ subject to
$P\geq 0$, was solved by projected Gauss--Seidel. On the mesh with $833$
nodes the two solutions agree to $3.7\cdot 10^{-13}$ in the maximum norm,
their free-node contact sets coincide, and their energies differ by
$2.2\cdot 10^{-16}$ in relative terms.

\subsection{Reynolds Convergence Verification}
\label{subsec:reynolds-convergence}

We verify the nodal mixed method on \(\Omega=(-1,1)^2\), with
\(d\equiv1\) and homogeneous Dirichlet data on the entire boundary.
Writing \(r^2=x^2+y^2\), take
\begin{align}
P(x,y) &= \bigl(\tfrac14-r^2\bigr)_+^2 \label{eq:manufactured-obstacle}\\
f(x,y) &= \max\{2-16r^2,-2\} \label{eq:manufactured-obstacle-load}
\end{align}
and \(\lambda=0\) for \(r<1/2\), \(\lambda=2\) for \(r>1/2\).
Inside the disk, \(-\Delta P=2-16r^2=f\); outside, \(P=0\) and
\(\lambda=-f\). Both \(P\) and its gradient vanish on the circular free
boundary, so the Kuhn--Tucker equations hold weakly without an interface
source. The solution belongs to \(H^2(\Omega)\), with a jump in its second
derivatives across the free boundary.

We divide each side of the square into \(n\) intervals and split each
square cell into two triangles. These meshes are uniformly shape regular,
with maximum diameter \(h=2\sqrt{2}/n\), and the circular free boundary
is not fitted to the mesh. Table~\ref{tab:reynolds-convergence} reports
errors of the raw two-dimensional finite element solution, without
averaging or other postprocessing. The load uses the solver's degree-three
triangle quadrature. Error integration uses tensor Gauss quadrature after
a Duffy transformation, with two additional uniform refinement levels on
triangles that may intersect the circle. Increasing the quadrature from
six to ten points in each transformed coordinate changes the reported
\(L_2\) and \(H^1\)-seminorm errors by at most \(4.9\cdot10^{-8}\) and
\(7.6\cdot10^{-6}\), respectively, in relative terms.

\begin{table}[htbp]
\centering
\begin{tabular}{rrrrrrr}
\toprule
\(n\) & elements & \(\|P-P_h\|_{\Omega}\) & rate & \(\vert P-P_h\vert_{H^1(\Omega)}\) & rate & its \\
\midrule
16  & 512   & \(2.334\cdot10^{-3}\) & ---  & \(3.970\cdot10^{-2}\) & ---  & 5 \\
32  & 2048  & \(5.510\cdot10^{-4}\) & 2.08 & \(2.012\cdot10^{-2}\) & 0.98 & 7 \\
64  & 8192  & \(1.458\cdot10^{-4}\) & 1.92 & \(1.020\cdot10^{-2}\) & 0.98 & 10 \\
128 & 32768 & \(3.510\cdot10^{-5}\) & 2.05 & \(5.124\cdot10^{-3}\) & 0.99 & 17 \\
\bottomrule
\end{tabular}
\caption{Raw two-dimensional Reynolds errors for
\eqref{eq:manufactured-obstacle} on uniformly refined square meshes.
The final column gives the active-set iteration count.}
\label{tab:reynolds-convergence}
\end{table}

The first-order gradient accuracy agrees with
Theorem~\ref{thm:reynolds-error}. The second-order \(L_2\) accuracy is a
property of this experiment. All four active-set runs converge with the
nodal sign and complementarity checks satisfied.

\subsection{Multiplier-Free Stabilised Reynolds Method}
\label{subsec:reynolds-stabilized-results}

We turn to the method \eqref{eq:reynolds-discrete-functional}--\eqref{eq:reynolds-stabilized-form} of
\cite{GuRaStVi18}, taken with $k=2$. By Theorem~\ref{thm:reynolds-stability} the parameter
must be scaled as $\gamma\vert_T = \gamma_0h_T^2$ and $\gamma_0$ kept below
a threshold, and the point of this subsection is to locate that threshold and
then to check the two discretisations against each other.
The largest generalised eigenvalue $\Lambda_{\max}$ of
Remark~\ref{rem:reynolds-parameter-threshold}, the reciprocal of the
threshold, is $231$, $290$ and $334$ on refinement levels $2$, $3$ and $4$,
with $161$, $705$ and $2945$ degrees of freedom. It is still growing slowly,
so the admissible range $\gamma_0<0.0030$ is read from the finest mesh.

Taking $\gamma_0 = 10^{-3}$, a factor three inside the bound, gives Table
\ref{tab:reynolds-method-comparison}. Above the bound the frozen systems lose their definiteness and
the iteration does not converge; at $\gamma_0 = 1$, for instance, the
pressure acquires excursions to $-0.12$ and the active set never settles.

\begin{table}[htbp]
\centering
\small
\begin{tabular}{rrrrrrrr}
\toprule
& & \multicolumn{2}{c}{nodal, $\mathbb{P}_1$} &
\multicolumn{3}{c}{stabilised, $\mathbb{P}_2$} & \\
ref. & elem. & its & $\max P_h$ & its & $\max P_h$ & $\min P_h$
& difference \\
\midrule
3 & 384 & 4 & 0.0407945 & 14 & 0.0419143 & $-1.3\cdot10^{-5}$ & $6.1\cdot 10^{-4}$ \\
4 & 1536 & 8 & 0.0418111 & 24 & 0.0419537 & $-1.8\cdot10^{-6}$ & $2.2\cdot 10^{-4}$ \\
5 & 6144 & 13 & 0.0419265 & 44 & 0.0419507 & $-7.9\cdot10^{-7}$ & $3.9\cdot 10^{-5}$ \\
\bottomrule
\end{tabular}
\caption{The two discretisations of the Reynolds model compared,
$\gamma_0 = 10^{-3}$. The last column is the difference at the vertices
shared by the two meshes.}
\label{tab:reynolds-method-comparison}
\end{table}

The decreasing difference supports convergence towards a common solution:
the difference at the shared vertices falls
by an order of magnitude over the range of meshes, and the centreline
profiles in Figure~\ref{fig:reynolds-centerline-pressure} are indistinguishable.
The active-set iteration of the stabilised method needs $14$, $24$ and $44$
iterations, roughly doubling under refinement, against $4$, $8$ and $13$
for the nodal method. The two methods differ in how
the constraint is met. The nodal method satisfies $P_h\geq 0$ exactly; the
stabilised method is consistent and imposes it only weakly, so a small
undershoot survives, of size $10^{-5}$ to $10^{-7}$ here and decreasing
under refinement.

\begin{figure}[htbp]
\centering
\includegraphics[width=0.78\textwidth]{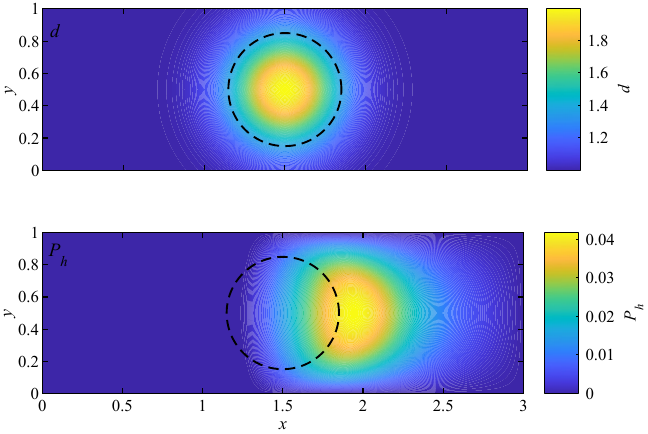}
\caption{The pit \eqref{eq:reynolds-pit-geometry}: film thickness $d$ above, scaled pressure
$P_h$ below, the broken line marking the circle of radius $r$ about the
centre of the pit. The lubricant moves in the positive $x$ direction. The
upstream half of the pit, where the gap diverges, is cavitated and carries
$P_h = 0$; the pressure builds over the converging half and peaks just
downstream of it.}
\label{fig:reynolds-pit-pressure}
\end{figure}

\begin{figure}[htbp]
\centering
\includegraphics[width=0.78\textwidth]{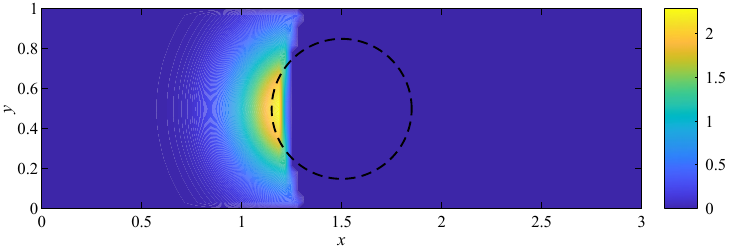}
\caption{The multiplier $\lambda_h$ for the pit of Figure~\ref{fig:reynolds-pit-pressure}. It
vanishes wherever the film is intact, and is supported on the cavitated
region, where it is the source that holds the pressure at zero. Note that
it concentrates against the free boundary rather than spreading over the
cavitated region.}
\label{fig:reynolds-pit-multiplier}
\end{figure}

\begin{figure}[htbp]
\centering
\includegraphics[width=0.7\textwidth]{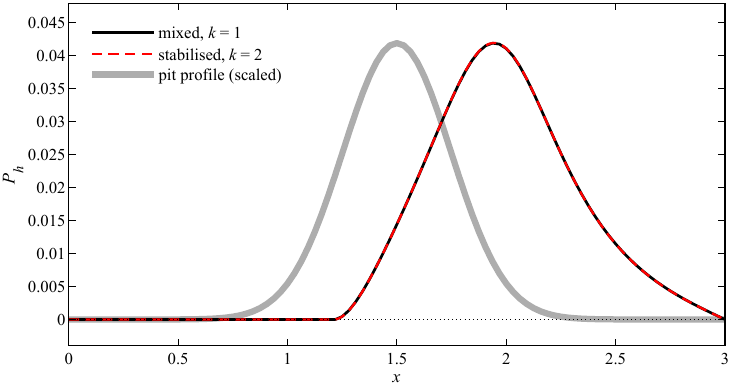}
\caption{Pressure along the centreline $y = 0.5$ for the two
discretisations of the Reynolds model, with the pit profile shown scaled
for reference.}
\label{fig:reynolds-centerline-pressure}
\end{figure}

\subsection{Two-Dimensional Stokes Verification}
\label{subsec:stokes-two-dimensional}

We take the channel $\Omega = (0,3)\times(0,1)$ with $\mu = 1$ and
$\bm{f} = \bm{0}$. On $\Gamma\suprm{D}$ we prescribe no slip along $y=0$
and $y=1$ and the inflow profile $\bm{g} = (y(1-y),0)$ along $x=0$;
$\Gamma\suprm{N}$ is the outflow $x=3$, left traction free. In the bulk the
solution is the Poiseuille flow $\bm{u} = (y(1-y),0)$ with $p = 6-2x$, but
that profile does not satisfy $\bm{\sigma}\cdot\bfn = \bm{0}$ at the
outflow, since the shear component of $\bm{\sigma}$ there is
$2\mu\varepsilon_{12} = 1-2y$, which does not vanish. The pressure is therefore drawn down near $x = 3$ and
cavitates in a thin layer against the outflow face, as
Figure~\ref{fig:stokes-channel} shows.

The velocity is approximated by the stabilised Crouzeix--Raviart element
with $\gamma_1 = 1$ and the pressure by piecewise constants. Table
\ref{tab:stokes-channel-refinement} shows the refinement behaviour. Since the
pressure is piecewise constant and has two traces on the mesh-aligned
centreline, it is read from the upper trace at the longitudinal cell
centres, over $x<2.5$, away from the corners where the no-slip walls meet
the traction-free end. The sampled peak pressures increase towards the Poiseuille inflow reference
value $6$ over these meshes, and the cavitated fraction is about $1.9$
per cent on the two finest meshes. The exact traction-free channel need not
have peak pressure $6$.

\begin{table}[htbp]
\centering
\small
\begin{tabular}{rrrrrrr}
\toprule
ref. & elements & d.o.f. & its & $\max p^h$ & cavitated fr.
& $\max_T\vert p_T(\nabla\cdot\bm{u}^h)_T\vert$ \\
\midrule
2 & 96   & 416   & 3 & 5.7271 & 0.0104 & $1.3\cdot10^{-15}$ \\
3 & 384  & 1600  & 2 & 5.7739 & 0.0208 & $2.6\cdot10^{-15}$ \\
4 & 1536 & 6272  & 4 & 5.8228 & 0.0189 & $8.3\cdot10^{-15}$ \\
5 & 6144 & 24832 & 3 & 5.8525 & 0.0187 & $2.1\cdot10^{-14}$ \\
\bottomrule
\end{tabular}
\caption{Stokes model in two dimensions, stabilised Crouzeix--Raviart with
$\gamma_1 = 1$ and $\gamma = 100$. The cavitated fraction is by area; the
last column is the elementwise complementarity residual.}
\label{tab:stokes-channel-refinement}
\end{table}

The complementarity conditions
hold to round-off, as Theorem~\ref{thm:stokes-exact-complementarity} says they must: $p_T = 0$ on the
cavitated elements and $(\nabla\cdot\bm{u}^h)_T = 0$ off them, with pressure and divergence nonnegative to numerical precision. The iteration counts do not grow, in contrast with Tables
\ref{tab:reynolds-nodal-refinement} and \ref{tab:reynolds-convergence}; the cavitated set is a thin layer against one
face rather than a region whose boundary has to travel.

The parameter $\gamma$ has disappeared from the method, again by Theorem
\ref{thm:stokes-exact-complementarity}. Varying it over eight orders of magnitude, from $10^{-4}$ to
$10^{4}$, at refinement level $4$ left the computed pressure unchanged to
$5.4\cdot 10^{-12}$ in relative $\ell_2$ norm, the cavitated set unchanged at
$29$ elements, and the iteration count unchanged at four. This is a
substantive gain over a continuous pressure space, for which the same
computation leaves a residual violation of the constraint of order
$\gamma^{-1}$; here there is none to leave.

\begin{figure}[htbp]
\centering
\includegraphics[width=0.85\textwidth]{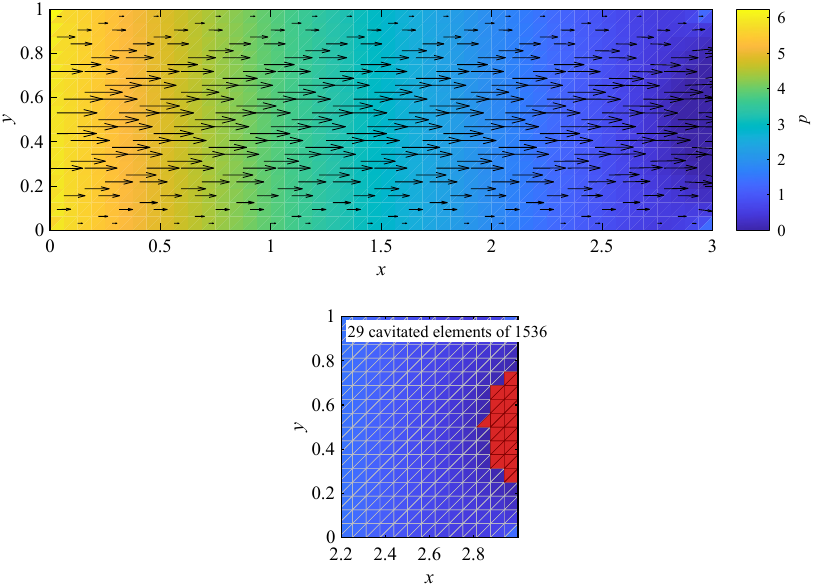}
\caption{The two-dimensional Stokes problem. Above, the pressure and the
velocity field; the flow enters with the Poiseuille profile and spreads as
it approaches the traction-free end. Below, the outflow region enlarged,
with the cavitated elements in red.}
\label{fig:stokes-channel}
\end{figure}

\subsection{Three-Dimensional Stokes Verification}
\label{subsec:stokes-three-dimensional}

We extrude the channel to the box $(0,3)\times(0,1)\times(0,1)$ and
solve \eqref{eq:stokes-discrete-momentum}--\eqref{eq:stokes-discrete-multiplier} on tetrahedra, again with the stabilised
Crouzeix--Raviart element, the velocity being nonconforming piecewise linear
with one degree of freedom per face and the pressure piecewise constant. The
jump penalty \eqref{eq:cr-stabilized-form} carries over unchanged, the faces of the
tetrahedra taking the place of the edges of the triangles, with
$h_F = (\vert T^+\vert+\vert T^-\vert)/(2\vert F\vert)$. The mesh is
obtained by dividing a structured grid of
hexahedra into six tetrahedra each by the Kuhn subdivision, which is
conforming. The channel boundary conditions of
Section~\ref{subsec:stokes-two-dimensional} are retained,
and on the two faces $z=0$ and $z=1$ we impose $u_z = 0$ and leave the
remaining components free. These faces form $\Gamma\suprm{S}$ in
\eqref{eq:stokes-symmetry-boundary}: the normal component is imposed
through its Crouzeix--Raviart face mean, while the tangential components
carry the natural zero traction. This is the componentwise essential
condition included in \eqref{eq:stokes-spaces} and
\eqref{eq:cr-velocity-space}, so the experiment lies within the boundary
framework analysed above. The faces are symmetry planes, and the exact
solution is the two-dimensional one of the previous subsection extruded in
$z$, independent of $z$ and with $u_z\equiv 0$. That makes the two
computations directly comparable.

The pressure is piecewise constant, so it has no unique trace on the
mesh-aligned line $y=z=0.5$. The peak in
Table~\ref{tab:stokes-three-dimensional} is therefore the maximum over the
one-cell tube surrounding that line, restricted to $x<2.5$. This fixed
geometric rule avoids a tetrahedron-search tie at the line itself.

\begin{table}[htbp]
\centering
\begin{tabular}{lrrrrrr}
\toprule
mesh & elements & d.o.f. & its & $\max p^h$ & cavitated fr.
& $\max\vert u_z\vert$ \\
\midrule
$9\times3\times3$ & 486 & 3780 & 2 & 6.5538 & 0.0123 & $6.5\cdot10^{-3}$ \\
$18\times6\times6$ & 3888 & 28728 & 3 & 6.0014 & 0.0183 & $2.9\cdot10^{-3}$ \\
$27\times9\times9$ & 13122 & 95256 & 3 & 5.9806 & 0.0171 & $1.9\cdot10^{-3}$ \\
$36\times12\times12$ & 31104 & 223776 & 3 & 5.9304 & 0.0186 & $1.4\cdot10^{-3}$ \\
\bottomrule
\end{tabular}
\caption{Stokes model in three dimensions, $\gamma = 100$. The cavitated
fraction is by volume; the last column measures the departure from the
two-dimensional solution, whose $u_z$ vanishes identically.}
\label{tab:stokes-three-dimensional}
\end{table}

The three-dimensional computation is consistent with the two-dimensional
behaviour, but the sampled peaks are not yet identical: $5.93$ against
$5.85$ on the finest meshes of Tables~\ref{tab:stokes-three-dimensional}
and \ref{tab:stokes-channel-refinement}. Both values are below the
Poiseuille reference value $6$. The three-dimensional peaks decrease under
refinement, while the two-dimensional ones increase. The two tables use different meshes and
sampling rules, so these values are neither an error bound nor a bracket
for the exact peak. The cavitated
fraction agrees more closely still, $0.0186$ against $0.0187$, and
cavitation occurs in both at the outflow face, as Figure~\ref{fig:stokes-three-dimensional}
shows. On the finest mesh the linear system of each active set is solved
by the iterated penalty method \cite{FG83}, with penalty parameter
$10^{4}$, in place of a sparse direct factorisation. On the
$9\times3\times3$ and $18\times6\times6$ meshes, where both solvers were
run, they give the same active sets and pressures that differ by less
than $2\cdot10^{-10}$. The complementarity conditions again hold to
round-off, the elementwise residual not exceeding $1.7\cdot 10^{-14}$ on
the three coarser meshes and $1.4\cdot 10^{-13}$ on the finest, and $\gamma$ again
cancels: varying it from $1$ to $10^{6}$ on the $18\times6\times6$ mesh left the
pressure unchanged to $3\cdot 10^{-14}$ and the cavitated set unchanged at
$71$ elements. Theorem~\ref{thm:stokes-exact-complementarity} makes no reference to the dimension.

The Kuhn subdivision is not itself symmetric in $z$, and the
solution inherits that: $\max\vert u_z\vert$ decreases only at first order under
refinement, from $6.5\cdot 10^{-3}$ to $1.4\cdot 10^{-3}$, where the exact
solution has $u_z\equiv 0$. Mixed-boundary singularities may affect the pressure
at the two edges $x=3$, $y\in\{0,1\}$, where a Dirichlet condition meets a
Neumann one, which is why the peak pressures of this subsection are read away
from the outflow, as in Table~\ref{tab:stokes-channel-refinement}.

\begin{figure}[htbp]
\centering
\includegraphics[width=0.95\textwidth]{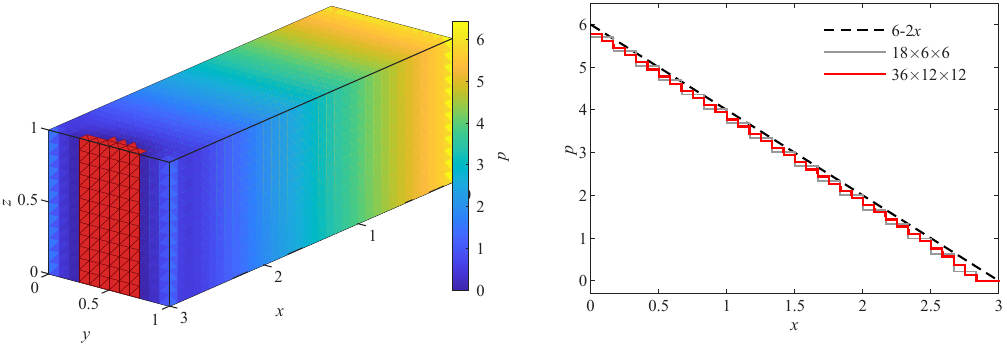}
\caption{Stokes flow with cavitation in three dimensions. Left, the
pressure on the boundary on the $36\times12\times12$ mesh, with the
cavitated elements at the outflow face marked. Right, the pressure along
the line $y = z = 0.5$ on the $18\times6\times6$ and $36\times12\times12$
meshes against the Poiseuille value $6-2x$. For $x<2.5$ the piecewise
constant pressure departs from it by at most $0.34$ and $0.23$,
respectively, and it vanishes near the outflow.}
\label{fig:stokes-three-dimensional}
\end{figure}

\subsection{Constitutive and Discretisation Effects}
\label{subsec:constitutive-discretization}

We use the film $G(x)<z<c$ between the stationary shaped surface
$G(x) = c-H(x)$ below, which dips into the workpiece where the pit is, and
the plane $z=c$ above, sliding with velocity $(V,0)$, the gap thickness
being
\begin{equation}\label{eq:common-pit-geometry}
H(x) = c\left(1+\delta\,{\rm e}^{-(x-x_{\rm p})^2/r^2}\right),
\qquad 0<x<L
\end{equation}
with $L = 24c$ and $x_{\rm p} = L/2$, and set $c = \mu = V = 1$. The slope
$\delta/r$ measures how steeply the pit cuts into the film. The Stokes
problem is solved in the cross section, with $\bm{u}=\bm{0}$ on the shaped
surface and $\bm{u}=(V,0)$ on the sliding plane. In this subsection and in
Section~\ref{subsec:solver-robustness} both ends are traction free.
Pressures are sampled on the mid-plane at $8001$ points on $[0,L]$ and
retained on the window $4<x<20$. A sample is cavitated when its element
belongs to the inactive pressure set or, for Taylor--Hood, when
$p^h\leq10^{-8}$; the cavity length is the sample spacing $0.003$ times the
number of cavitated samples.

Section~\ref{sec:stokes-cavitation}
argued that the constitutive law matters, and Section~\ref{subsec:stokes-discretization} that the
stabilisation \eqref{eq:cr-stabilized-form} is what allows that law to be used on the pair
adopted here. The first is a question about the model, and is settled by
holding the discretisation fixed and changing the law; the second concerns stability of the chosen discrete viscous form. We
compare two discretisations of the deviatoric formulation and, separately,
a full-gradient CR formulation without a jump penalty.

For the first, we solve the same problem twice on the same Taylor--Hood
discretisation, with the same meshes and the same $\gamma$, so that the
element of this paper and its stabilisation stay out of the experiment,
changing only the first
Lam\'e parameter: $-2\mu/3$, which is \eqref{eq:deviatoric-cauchy-stress}, against $0$, which is
the customary $\bm{\sigma} = 2\mu\bfeps(\bm{u})-p\bfI$ that Section
\ref{sec:stokes-cavitation} argues against. Any difference is then the constitutive law
alone. Table~\ref{tab:constitutive-law-comparison} reports it. The peak pressure is insensitive,
moving by under one per cent on all four pits; the cavitated length is not,
with ratios $1.09$, $1.05$, $1.27$, and $2.23$ in order of increasing
steepness. The largest discrepancy occurs at the steepest pit, where \eqref{eq:deviatoric-cauchy-stress} gives
$0.498$ against $1.113$. The customary law gives the larger cavity in all
four tests; Section~\ref{sec:stokes-cavitation} explains the change in
mechanical pressure but gives no ordering of the cavity sets. The
mechanical pressure itself shows the difference directly. On the stabilised
Crouzeix--Raviart discretisation of the steepest pit, mesh $256\times16$,
the customary law gives
$\bar p = p-\tfrac{2\mu}{3}\nabla\cdot\bm{u}<0$ on all $676$ cavitated
elements, down to $-0.64$, and a cavity length of $1.314$, whereas the
deviatoric law gives $\bar p = p = 0$ on its $492$ cavitated elements and
a cavity length of $0.843$.

\begin{table}[htbp]
\centering
\small
\begin{tabular}{rrrrrrr}
\toprule
& \multicolumn{2}{c}{$\max p$} & & \multicolumn{2}{c}{cavitated length} & \\
$\delta/r$ & \eqref{eq:deviatoric-cauchy-stress} & $2\mu\bfeps-p\bfI$ & difference
& \eqref{eq:deviatoric-cauchy-stress} & $2\mu\bfeps-p\bfI$ & ratio \\
\midrule
0.0625 & 3.3738 & 3.3566 & 0.51\% & 1.935 & 2.106 & 1.09 \\
0.25   & 5.4539 & 5.4336 & 0.37\% & 2.706 & 2.847 & 1.05 \\
0.5    & 3.7559 & 3.7348 & 0.56\% & 1.632 & 2.073 & 1.27 \\
1.0    & 2.1466 & 2.1253 & 0.99\% & 0.498 & 1.113 & 2.23 \\
\bottomrule
\end{tabular}
\caption{The two constitutive laws on one Taylor--Hood discretisation of
the pits \eqref{eq:common-pit-geometry}, mesh
$256\times 16$ and $\gamma = 100$: the deviatoric law
\eqref{eq:deviatoric-cauchy-stress}, with first Lam\'e parameter
$-2\mu/3$, and $\bm{\sigma}=2\mu\bfeps(\bm{u})-p\bfI$, with first Lam\'e
parameter zero. Only the law differs.}
\label{tab:constitutive-law-comparison}
\end{table}

We turn to the element. The pair used throughout is nonconforming and its
pressure is only piecewise constant, so it is worth checking against a
standard conforming one; and the stabilisation \eqref{eq:cr-stabilized-form} is worth its
place only if omitting it does damage. Table~\ref{tab:discretization-comparison} therefore reports
three computations of the steepest pit: the proposed method, a
Taylor--Hood discretisation of the same deviatoric formulation, and CR with
a full-gradient viscous form and no jump penalty. Taylor--Hood uses
continuous piecewise quadratic velocity and continuous piecewise linear
pressure. The full-gradient comparison changes both the viscous form and
the natural traction condition and is not a controlled penalty-removal
experiment.

\begin{table}[htbp]
\centering
\small
\setlength{\tabcolsep}{4.5pt}
\begin{tabular}{rlrrrrrr}
\toprule
mesh & pair & d.o.f. & its & $\max p$ & cav.\ length & $\min p^h$
& $\max_T\vert p_T(\nabla\cdot\bm{u}^h)_T\vert$ \\
\midrule
$128\times 8$ & proposed         & 8464   & 8   & 2.1537 & 0.75 & $0$ & $6.0\cdot10^{-15}$ \\
              & full-gradient CR  & 8464   & 9   & 1.9870 & 5.94 & $0$ & $3.4\cdot10^{-15}$ \\
              & Taylor--Hood      & 9899   & 12  & 2.1556 & 0.45 & $-5.3\cdot10^{-4}$ & --- \\
$256\times 16$& proposed         & 33312  & 9   & 2.1473 & 0.84 & $0$ & $1.0\cdot10^{-14}$ \\
              & full-gradient CR  & 33312  & 9   & 2.0934 & 5.94 & $-5.9\cdot10^{-15}$ & $1.7\cdot10^{-14}$ \\
              & Taylor--Hood      & 38227  & 13  & 2.1466 & 0.50 & $-2.8\cdot10^{-4}$ & --- \\
$512\times 32$& proposed         & 132160 & 10  & 2.1451 & 0.80 & $0$ & $3.0\cdot10^{-14}$ \\
              & full-gradient CR  & 132160 & 9   & 2.1239 & 5.89 & $-1.5\cdot10^{-14}$ & $3.3\cdot10^{-14}$ \\
              & Taylor--Hood      & 150179 & 15  & 2.1452 & 0.69 & $-1.8\cdot10^{-4}$ & --- \\
\bottomrule
\end{tabular}
\caption{The proposed deviatoric CR method, full-gradient CR without a
jump penalty, and deviatoric Taylor--Hood on the steepest pit, $\delta/r=1$,
of \eqref{eq:common-pit-geometry}, with $\gamma = 100$.
Pressures and cavity lengths use the same mid-plane window. The two finer
full-gradient rows are stopped by the sign rule of
Section~\ref{subsec:primal-dual-active-set}.}
\label{tab:discretization-comparison}
\end{table}

The proposed method agrees with Taylor--Hood to four digits in the peak
pressure on the finest mesh, $2.1451$ against $2.1452$, on two
discretisations sharing only the continuous formulation and on some twelve
per cent fewer degrees of freedom. The cavitated lengths are of the same
size but have not settled: $0.75$, $0.84$ and $0.80$ against $0.45$, $0.50$
and $0.69$ on the three meshes. On the finest mesh the gentler pit
$\delta/r = 0.5$ gives $3.7548$ against $3.7548$ for the peak pressure and
$2.06$ against $1.89$ for the cavitated length. Where the
two part company is the constraint: the Crouzeix--Raviart pressure is
nonnegative to round-off and the elementwise conditions hold to
$10^{-14}$, by Theorem~\ref{thm:stokes-exact-complementarity}, against an undershoot of some $3\cdot10^{-4}$ for Taylor--Hood on the
intermediate mesh at the same $\gamma$. The continuous pressure space does
not guarantee exact nonnegativity through this multiplier equation.

The full-gradient CR rows give much larger cavity lengths. The difference
cannot be attributed to the jump penalty alone, since the viscous form and
its natural boundary condition have also changed.

The penalty parameter itself
is not delicate: $\gamma_1$ between $0.1$ and $2$ moves the peak pressure
by under two per cent. Its presence is essential: at $\gamma_1 = 0$
with the strain-based form retained the computation degenerates outright, the
pressure collapsing to $1.50$ with no cavitation at all, consistent with the loss of strain control discussed in
Section~\ref{subsec:stokes-discretization}. Large $\gamma_1$
locks the element towards its conforming counterpart and is no better: at
$\gamma_1 = 100$ the peak pressure is $3.50$.

The solver behaviour also differs. For the proposed method the active set
repeats after eight, nine and ten iterations on the three meshes. For
full-gradient CR it repeats after nine iterations on the coarsest mesh, but
on the two finer meshes it never repeats: about $3200$ of $8192$ elements
remain cavitated, and the set changes only by elements on which both the
pressure and the divergence vanish to rounding error. Every such frozen
solution satisfies the Kuhn--Tucker conditions to rounding error, and the
sign rule stops the iteration after nine iterations on both meshes; on
$8192$ elements the result is the quadratic programming reference of
Section~\ref{subsec:solver-robustness}.

\subsection{Reynolds--Stokes Model Comparison}
\label{subsec:model-comparison}

The Reynolds model is not well suited to large variations in the geometry
of the lubrication layer, which is why we carry the Stokes model alongside
it. Comparisons of lubrication theory with Stokes flow have been made for
geometries with flow separation, where the loss of accuracy at steep
gradients is documented \cite{DeFa25}, and extended-lubrication and fast
Reynolds-solver studies characterise pressure and velocity errors over
textured geometries \cite{DeFa26Ext,DeFa26Fast}. Those works do not impose
the cavitation inequality. Both of our models enforce $p\geq 0$ and nothing
else, the Swift--Stieber condition, and neither conserves mass across the
cavitated zone.

We use the four pits \eqref{eq:common-pit-geometry} of
Table~\ref{tab:constitutive-law-comparison}. The Reynolds problem is solved
on a strip with $P$ prescribed only at the two ends, so that its solution
is the one-dimensional one, and the models are compared through
$p = 6\mu VP/c^2$. The Stokes end conditions must be chosen with care. For
a flat channel of gap $c$ the fully developed Couette solution is
\begin{equation}
\bm{u}(x,z)=(Vz/c,0),\qquad p=p\subrm{b} \label{eq:flat-couette-reference}
\end{equation}
for a constant ambient pressure $p\subrm{b}\geq0$. At a vertical end
$\sigma_{xx}=-p\subrm{b}$ but $\sigma_{zx}=\mu V/c$, so traction-free ends
do not admit Couette flow, even at $p\subrm{b}=0$. We therefore use the
pressure-normal-flow ends \eqref{eq:stokes-end-boundary}, $u_z=0$ and
$\bm{n}\cdot\bm{\sigma}(\bm{u},p)\bm{n}=-p\subrm{b}$, which admit
\eqref{eq:flat-couette-reference} and permit the normal flux required by
Remark~\ref{rem:outflow-flux}. For straight vertical ends and sufficiently
regular traces, $u_z=0$ gives $\partial_zu_z=0$, hence
$\partial_xu_x=\nabla\cdot\bm{u}$, and \eqref{eq:deviatoric-cauchy-stress}
gives
\begin{align}
p &= p\subrm{b}+\tfrac43\mu\,\nabla\cdot\bm{u} \label{eq:pressure-end-pressure}\\
0 &= p\,\nabla\cdot\bm{u}=p\subrm{b}\,\nabla\cdot\bm{u}+\tfrac43\mu(\nabla\cdot\bm{u})^2 \label{eq:pressure-end-complementarity}
\end{align}
so nonnegative divergence and $p\subrm{b}$ imply $\nabla\cdot\bm{u}=0$ and
$p=p\subrm{b}$ at the ends, the Reynolds end condition. On $128\times8$ and
$256\times16$ flat-channel meshes with $p\subrm{b}=0$ and $1$ the method
reproduces \eqref{eq:flat-couette-reference} to within $10^{-12}$. We take
$p\subrm{b}=0$ below.

The longitudinal spacing is $24/256$, with $16$ Stokes layers across the
mapped gap and four Reynolds strip layers. The Stokes problems are solved
by quadratic programming followed by working-set refinement. Pressures are
sampled on the mid-plane with spacing $0.000375$ in the window $4<x<20$,
using the upper trace of the discontinuous Stokes pressure. A sample is
cavitated when both Reynolds nodes of its interval belong to the
multiplier-active set, or when its Stokes element belongs to the inactive
pressure set. Every cavity reaches the upstream edge of the window, so we
report the reformation front $\xi=x\subrm{f}-x\subrm{p}$, with
$x\subrm{f}$ the rightmost cavitated sample, rather than a length. Fronts
lie on element edges and are resolved to one longitudinal cell, $0.094$.

\begin{table}[htbp]
\centering
\small
\begin{tabular}{rrrrrrrr}
\toprule
& & & \multicolumn{2}{c}{$\max p$} & & \multicolumn{2}{c}{front $\xi$} \\
$\delta$ & $r$ & $\delta/r$ & Reynolds & Stokes & profile diff. & Reynolds & Stokes \\
\midrule
0.5 & 8 & 0.0625 & 3.3653 & 3.3631 & 0.63\% & $-6.094$ & $-5.813$ \\
1.0 & 4 & 0.25   & 5.4940 & 5.4465 & 1.34\% & $-5.250$ & $-5.063$ \\
1.0 & 2 & 0.5    & 3.8098 & 3.7366 & 2.48\% & $-3.188$ & $-3.094$ \\
1.0 & 1 & 1.0    & 2.2158 & 2.1289 & 4.52\% & $-1.781$ & $-1.969$ \\
\bottomrule
\end{tabular}
\caption{The two models for the pits \eqref{eq:common-pit-geometry} with
pressure-normal-flow Stokes ends on $[0,24]$. The profile difference is
$100\|p_{\rm S}-p_{\rm R}\|_{L_\infty}/\max p_{\rm R}$ on the window
$4<x<20$, and $\xi$ is the reformation front relative to the pit centre.}
\label{tab:reynolds-stokes-comparison}
\end{table}

Table~\ref{tab:reynolds-stokes-comparison} and
Figure~\ref{fig:reynolds-stokes-pressure} show that with these ends the two
models agree closely. The profile difference grows with the slope, from
$0.63$ to $4.5$ per cent, and the fronts differ by one to three cells. The
Stokes front lies downstream of the Reynolds one for $\delta/r\leq 0.5$ and
upstream of it for $\delta/r=1$.
Figure~\ref{fig:stokes-reversed-flow} shows reversed flow in the depth of
the steepest pit, which the gap-averaged Reynolds model cannot represent.

\begin{figure}[!htb]
\centering
\includegraphics[width=0.6\textwidth]{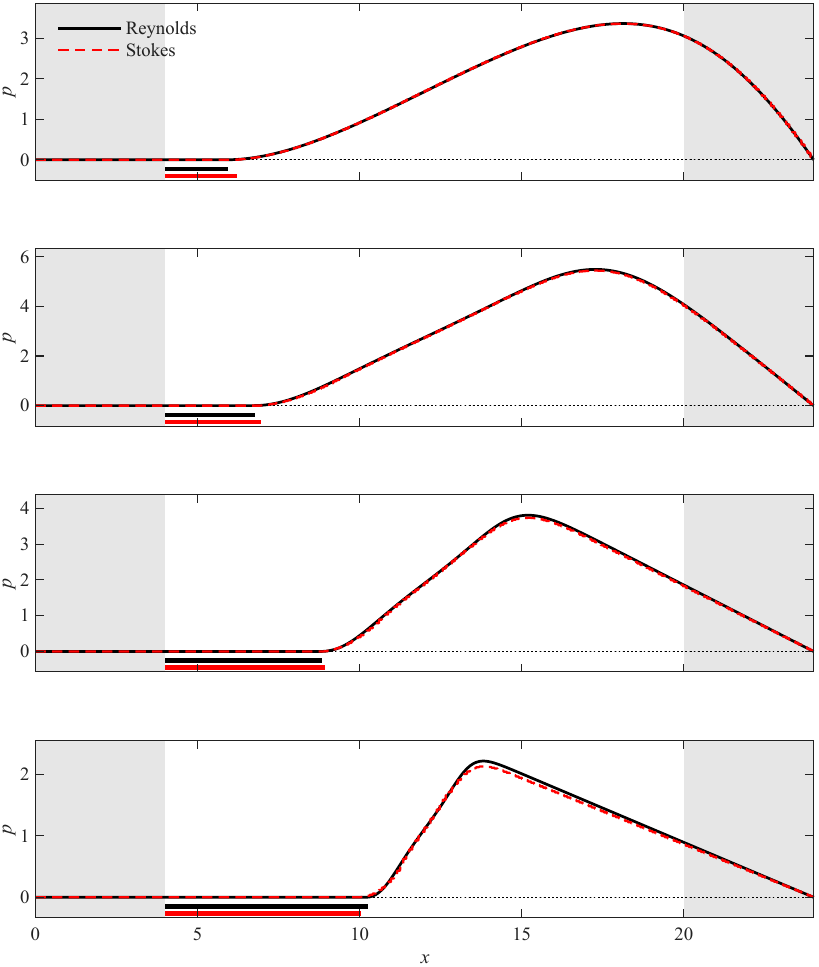}
\caption{Pressure on the mid plane for the four pits of
Table~\ref{tab:reynolds-stokes-comparison}, with the steepness increasing
downwards: $\delta/r = 0.0625$, $0.25$, $0.5$ and $1$. The shaded strips lie
outside the observation window. The bars below each profile mark the
cavitated samples, black for the Reynolds model and red for the Stokes one.}
\label{fig:reynolds-stokes-pressure}
\end{figure}

\begin{figure}[!htb]
\centering
\includegraphics[width=0.9\textwidth]{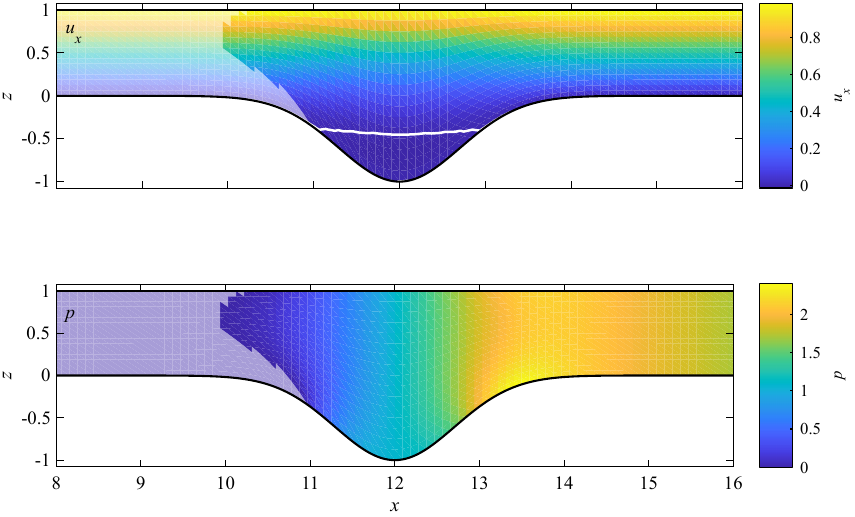}
\caption{The Stokes solution near the steepest pit, $\delta/r = 1$, with
pressure-normal-flow ends and the sliding surface above moving to the
right. Above, the horizontal velocity, the white line being the contour
$u_x = 0$ that bounds the reversed flow in the pit. Below, the pressure.
The cavitated region is lightened in both.}
\label{fig:stokes-reversed-flow}
\end{figure}

For the steepest pit an independent one-dimensional Reynolds calculation
supplies a front reference. In the positive-pressure region, integrating the
Reynolds equation and imposing $P'(x\subrm{f})=0$ gives
$H^3P'=H-H(x\subrm{f})$, so the front satisfies
\begin{equation}
\int_{x_{\mathrm f}}^{b}
\frac{H(x)-H(x_{\mathrm f})}{H(x)^3}\,\mathrm{d}x=0 \label{eq:reynolds-front-quadrature}
\end{equation}
where $b$ is the outlet location. Quadrature and a scalar root solve give
$\xi_{\mathrm R}=-1.81427$ for $b=24$. We halve the longitudinal step in
$6\leq x\leq18$ and double the number of transverse layers twice, add a
globally refined mesh, and shift the finest local grid by half a cell;
Table~\ref{tab:pressure-end-refinement} reports the results.

\begin{table}[htbp]
\centering
\small
\begin{tabular}{lrrrrr}
\toprule
mesh & Stokes elements & local \(h_x\) & \(\xi_{\mathrm R}\) & \(\xi_{\mathrm S}\) & Stokes peak \\
\midrule
reference & 8192   & \(3/32\)  & \(-1.7814\) & \(-1.9689\) & 2.128880 \\
local     & 24576  & \(3/64\)  & \(-1.8283\) & \(-1.9689\) & 2.125899 \\
local     & 81920  & \(3/128\) & \(-1.8050\) & \(-1.9689\) & 2.125309 \\
global    & 131072 & \(3/128\) & \(-1.8050\) & \(-1.9689\) & 2.125334 \\
shifted local & 82048 & \(3/128\) & \(-1.8166\) & \(-1.9573\) & 2.125267 \\
\bottomrule
\end{tabular}
\caption{Pressure-normal-flow end conditions on \([0,24]\). Stokes
uses 16, 32, and 64 transverse layers; the Reynolds strip is refined
correspondingly. The global row also refines the outer longitudinal cells; the final row
shifts the finest local grid by half a cell, retaining the refinement
interval endpoints.}
\label{tab:pressure-end-refinement}
\end{table}

The sampled Stokes front is unchanged on the nested meshes and on the
globally refined one, while the finite element Reynolds front approaches the
quadrature reference, with final error $0.0093$. The half-cell shift moves
the Stokes front by $0.0116$ and brings the Reynolds front within $0.0024$ of
its reference. The fine-grid Stokes fronts thus lie $0.143$ to $0.155$ gap
heights upstream of the one-dimensional Reynolds front on this domain. This
range is an observed grid sensitivity, not an error bound.

\begin{table}[htbp]
\centering
\small
\begin{tabular}{rrrrrrrr}
\toprule
\multicolumn{2}{c}{domain} & \multicolumn{2}{c}{Reynolds} &
\multicolumn{2}{c}{Stokes, traction-free} & \multicolumn{2}{c}{Stokes, pressure-normal-flow} \\
left & right & $\max p$ & $\xi$ & $\max p$ & cav.\ length & $\max p$ & $\xi$ \\
\midrule
0   & 24 & 2.2158 & $-1.781$ & 2.1473 & 0.843 & 2.1289 & $-1.969$ \\
$-12$ & 24 & 2.2158 & $-1.781$ & 2.1471 & 1.218 & 2.1289 & $-1.969$ \\
0   & 36 & 2.4018 & $-1.969$ & 2.3488 & 0.000 & 2.3353 & $-2.438$ \\
$-12$ & 36 & 2.4018 & $-1.969$ & 2.3474 & 0.000 & 2.3353 & $-2.438$ \\
$-36$ & 60 & 2.5115 & $-2.156$ & 2.4659 & 0.000 & 2.4561 & --- \\
\bottomrule
\end{tabular}
\caption{End sensitivity for the steepest pit, with identical meshes on
$[0,24]$ and cells appended beyond it, and the fixed window $4<x<20$. The
traction-free cavity length uses the sampling of
Section~\ref{subsec:constitutive-discretization}; a dash or zero length
means that no mid-plane sample in the window is cavitated.}
\label{tab:domain-length-sensitivity}
\end{table}

Table~\ref{tab:domain-length-sensitivity} shows how strongly the Stokes
cavity depends on the ends. With traction-free ends the windowed Stokes
cavity length on $[0,24]$ is $0.843$, against $6.219$ for Reynolds, a ratio
of $7.4$. It grows by $44$ per cent when only the inlet is moved and leaves
the window when the outlet is moved. With pressure-normal-flow ends the
inlet location has no effect. Moving the outlet shifts both fronts upstream,
the Reynolds front by two cells and the Stokes front by five on $[0,36]$,
and on $[-36,60]$ the Stokes mid-plane cavity leaves the window, whose
minimum sampled pressure is $4.6\cdot10^{-3}$. Cavity predictions must
therefore be reported together with the end conditions, the end locations
and the observation window.

\subsection{Solver Robustness}
\label{subsec:solver-robustness}

The full-gradient Crouzeix--Raviart computation of the steepest pit,
without a jump penalty and with traction-free ends, is the case in which
the set of the active-set iteration does not repeat.
Table~\ref{tab:smoothed-newton-performance} compares the active-set
iteration stopped by the sign rule with the smoothed Newton method of
Section~\ref{subsec:smoothed-central-path} at $s=10^{-8}$, undamped and
started from the boundary data and $p^h=0$. For the smoothed solutions a
sample is cavitated when $p_T<(s/\gamma)^{1/2}$: on the central path $p_T$
and $(\nabla\cdot\bm{u}^h)_T$ are both positive with product $s/\gamma$, so
the two populations are separated by the geometric mean.

\begin{table}[htbp]
\centering
\small
\begin{tabular}{rrrrrrr}
\toprule
& \multicolumn{3}{c}{active set, sign rule} & \multicolumn{3}{c}{smoothed Newton, $s=10^{-8}$} \\
elements & its & $\max p^h$ & cav.\ length & its & $\max p^h$ & cav.\ length \\
\midrule
2048  & 8 & 1.98695 & 5.937 & 11 & 1.98695 & 2.439 \\
8192  & 9 & 2.09338 & 5.937 & 12 & 2.09338 & 2.346 \\
32768 & 9 & 2.12389 & 5.889 & 12 & 2.12389 & 2.343 \\
\bottomrule
\end{tabular}
\caption{The steepest pit, $\delta/r=1$, of \eqref{eq:common-pit-geometry},
with the full-gradient Crouzeix--Raviart pair, no jump penalty, traction-free
ends and $\gamma=100$. The active-set iteration uses the sign rule with
$\tau=10^{-12}$, and the Newton residual tolerance is $10^{-10}$.}
\label{tab:smoothed-newton-performance}
\end{table}

An independent reference is obtained by assembling the full-gradient
matrices separately and solving the convex velocity problem by quadratic
programming with working-set refinement; its formulation and tolerance
checks are included with the reproduction materials. On $2048$ and $8192$
elements the references have windowed peak pressures $1.98695238$ and
$2.09338476$ and cavity length $5.937$, and the active-set solutions agree
with them to all eight digits of the peak and in every cavitated sample.

The smoothed Newton method converged in every one of $48$ computations,
over four meshes, six values of $s$ from $10^{-2}$ to $10^{-12}$ and the
two steeper pits, with a superlinear residual decrease of observed order
$1.37$ to $1.59$; this is consistent with the remainder estimate of
\cite{HaLaLa25}, in which the radius of quadratic convergence shrinks like
$s^{1/2}$. Its pressure is accurate, but its cavity is not.

\begin{table}[htbp]
\centering
\small
\begin{tabular}{rrrrr}
\toprule
\(s\) & Newton & cav.\ length & length error (\%) & relative \(L_2\) pressure error \\
\midrule
\(10^{-8}\)  & 14 & 2.346 & 60.49 & \(4.37\cdot10^{-5}\) \\
\(10^{-10}\) & 13 & 3.189 & 46.29 & \(3.75\cdot10^{-6}\) \\
\(10^{-12}\) & 13 & 3.939 & 33.65 & \(3.16\cdot10^{-7}\) \\
\(10^{-14}\) & 13 & 4.689 & 21.02 & \(2.61\cdot10^{-8}\) \\
\(10^{-16}\) & 12 & 5.439 & 8.39  & \(2.09\cdot10^{-9}\) \\
\(10^{-18}\) & 10 & 5.937 & 0.00  & \(1.62\cdot10^{-10}\) \\
\(10^{-20}\) & 9  & 5.937 & 0.00  & \(1.19\cdot10^{-11}\) \\
\bottomrule
\end{tabular}
\caption{Smoothed steep-pit solutions compared with the unsmoothed QP
reference on 8192 elements, with \(\gamma=100\) and Newton residual
tolerance \(10^{-14}\). Pressure errors use the element-area-weighted
\(L_2\) norm over the entire domain. Cavity diagnostics use the same
mid-plane window and sampling step \(0.003\) as
Table~\ref{tab:smoothed-newton-performance}.}
\label{tab:smoothing-qp-reference}
\end{table}

Table~\ref{tab:smoothing-qp-reference} separates pressure accuracy from
cavity accuracy. At $s=10^{-12}$ the relative pressure error is
$3.2\cdot10^{-7}$ while the cavity is $34$ per cent too short, and the
sampled cavity set agrees with the reference only for $s\leq10^{-18}$, where
the central-path products are below rounding error. For this problem the
smoothed method is therefore no substitute for the active-set iteration
with the sign rule. For the nodal Reynolds method $\gamma$ cancels and the
active-set iteration reached a fixed point in every computation reported
here, so smoothing would only reintroduce a parameter. Further tolerance
checks are included with the reproduction materials.


\section{Concluding Remarks}
\label{sec:conclusions}

We have proposed augmented Lagrangian finite element methods for cavitation
in the Reynolds and Stokes models in which the discrete complementarity
conditions hold exactly and the augmentation parameter drops out. At the
continuous level both formulations are equivalent to their Kuhn--Tucker
systems. For Stokes flow the pressure multiplier also requires a part of the
boundary on which the normal velocity is free; it makes the divergence
surjective onto $L_2(\Omega)$ and supplies the constraint qualification, and
without it the problem reduces to incompressible Stokes flow.

Complementarity becomes exact when the multiplier equation can be tested
pointwise. For the Reynolds model nodal quadrature achieves this with
continuous piecewise linear pressure and multiplier, and the classical first-order
energy estimate follows. For the Stokes model the Crouzeix--Raviart
velocity with piecewise constant pressure gives an elementwise constant
multiplier residual, and the identity $\nabla\cdot\vec{V}^h=Q^h$ gives
nondegenerate pressure coupling. Unlike the customary incompressible stress,
which places a dilating cavitated point under tension, the deviatoric
stress makes the constrained scalar the mechanical pressure. The
jump-stabilised Crouzeix--Raviart method is stable for it in two and three
dimensions, the latter by the trace-free Korn estimate of the appendix, and
since the interpolant of a feasible velocity is discretely feasible, the
exact constraint also gives a first-order error estimate for velocity and
pressure.

The computations show that close pressure profiles do not imply close
cavity predictions. With traction-free ends the Stokes cavity of the
steepest pit is a seventh of the Reynolds one and changes with the domain
length. With ends calibrated against flat Couette flow the reformation
fronts of the two models lie within three longitudinal cells of each other
for all four pits, and on refined meshes the Stokes front of the steepest
pit lies about $0.15$ gap heights upstream of the one-dimensional Reynolds
reference. The active-set iteration can fail to repeat its set on
degenerate problems although its iterates have converged, and a stopping
rule on the sign conditions then terminates it. A smoothed Newton method
converges as well, but its cavity is accurate only for very small smoothing
parameters.

Extensions under consideration include estimates for the position of the
free boundary and mass-conserving cavitation models of Elrod--Adams type,
which require coupling the pressure constraint to saturation and transport.

\appendix

\section{A Boundary-Localised Piecewise-Linear Trace-Free Korn Estimate}
\label{sec:broken-korn-appendix}

This appendix
supplies a direct boundary-localised estimate for the piecewise linear
fields used by the method \eqref{eq:cr-stabilized-form}; the constant is uniform in $h$, and no lower
bound on $\gamma_1$ appears. The general three-dimensional discrete
trace-free Korn framework is due to Williams and Hong \cite{WiHo24}. Their
results cover piecewise $H^1$ and $H^2$ fields using projected jumps and an
auxiliary seminorm. The argument below is a CR/$\mathbb{P}_1$
specialisation with the stronger full jump norm and with boundary control
placed directly on $\Gamma\suprm{D}$.
Throughout, $n=3$, $\Omega$ is a bounded connected Lipschitz domain, the
mesh is a shape-regular tetrahedral partition of $\Omega$ with
$\Gamma\suprm{D}$ a union of its faces, every vertex patch is face
connected, and the essential boundary conditions eliminate every global
conformal Killing field. A relatively open planar patch in
$\Gamma\suprm{D}$ is sufficient for the latter condition. The symbol
$\nabla_h$ denotes the elementwise gradient, and the face sums run
over the interior faces and those on $\Gamma\suprm{D}$, with
$\jump{\bm{v}}$ the trace on the latter, as in \eqref{eq:cr-stabilized-form}. We
abbreviate the square of the norm in \eqref{eq:cr-coercivity} with $\bfeps$
replaced by the trace-free strain,
\begin{equation}
\Xi(\bm{v})^2 := \sum_{T\in\mathcal{T}^h}
\|\bfeps^{\rm dev}(\bm{v})\|_T^2
+\sum_{F}h_F^{-1}\|\jump{\bm{v}}\|_F^2
\label{eq:broken-trace-free-energy}
\end{equation}

In three dimensions the kernel of $\bfeps^{\rm dev}$ on a connected
domain is the ten-dimensional space of conformal Killing fields
\begin{equation}\label{eq:conformal-killing-fields}
\bm{w}(x) = \bm{a}+\bm{b}\times x+\vartheta\,x
+2(\bm{\zeta}\cdot x)\,x-\vert x\vert^2\bm{\zeta}
\end{equation}
see \cite{Da06}. The linear ones, $\bm{\zeta}=\bm{0}$, are the similarity
fields, a seven-dimensional space, and it is only these that can occur
elementwise in a piecewise linear function; this is what makes the
piecewise linear case elementary.

\begin{lem}\label{lem:conformal-field-plane}
A conformal Killing field vanishing on an open subset of a plane vanishes
identically.
\end{lem}

\begin{proof}
The restriction of \eqref{eq:conformal-killing-fields} to a plane is polynomial in two
variables, so vanishing on an open subset of the plane is vanishing on
all of it. Since $\bfeps^{\rm dev}$ transforms tensorially under rigid
motions, the class \eqref{eq:conformal-killing-fields} is invariant under them, and we may
take the plane to be $x_3=0$. The quadratic part of the third
component of \eqref{eq:conformal-killing-fields} there is $-(x_1^2+x_2^2)\zeta_3$, which forces
$\zeta_3=0$; the quadratic part of the first component is then
$\zeta_1x_1^2+2\zeta_2x_1x_2-\zeta_1x_2^2$, which forces $\zeta_1=\zeta_2=0$, and the field
is linear, $\bm{w}(x) = \bm{a}+\mathsf{R}x$ with
$\mathsf{R}x = \bm{b}\times x+\vartheta\,x$. Vanishing on the plane puts its two
tangent directions in the kernel of $\mathsf{R}$, while the eigenvalues
of $\mathsf{R}$ are $\vartheta$ and $\vartheta\pm i\vert\bm{b}\vert$; a two-dimensional
kernel requires zero to be an eigenvalue of algebraic multiplicity at
least two, which forces $\vartheta=0$ and $\bm{b}=\bm{0}$, whence
$\mathsf{R}=0$, and evaluating at any point of the plane,
$\bm{a}=\bm{0}$.
\end{proof}

We assume of $\Gamma\suprm{D}$, beyond positive measure, that no nonzero
conformal Killing field vanishes on it. By Lemma~\ref{lem:conformal-field-plane} this
holds whenever $\Gamma\suprm{D}$ contains a flat piece of positive area,
as it does in every geometry computed in this paper.

\begin{lem}\label{lem:continuous-trace-free-korn}
Under this assumption there is a constant $C$, depending on $\Omega$ and
$\Gamma\suprm{D}$, such that
\begin{equation}\label{eq:continuous-trace-free-korn}
\|\nabla\bm{u}\|_\Omega^2\leq
C\left(\|\bfeps^{\rm dev}(\bm{u})\|_\Omega^2
+\|\bm{u}\|_{L_2(\Gamma\suprm{D})}^2\right)
\qquad\forall\,\bm{u}\in[H^1(\Omega)]^3
\end{equation}
\end{lem}

\begin{proof}
The inequality
$\|\bm{u}\|_{H^1(\Omega)}\leq
C(\|\bfeps^{\rm dev}(\bm{u})\|_\Omega+\|\bm{u}\|_\Omega)$, proved for
bounded Lipschitz domains in \cite{Da06}, together with
the finite dimensionality of the kernel \eqref{eq:conformal-killing-fields} and the
compactness of $H^1(\Omega)$ in $L_2(\Omega)$, gives by the lemma of
Peetre and Tartar the quotient form
\begin{equation}\label{eq:conformal-killing-quotient}
\min_{\bm{w}\in CK}\|\bm{u}-\bm{w}\|_{H^1(\Omega)}
\leq C\,\|\bfeps^{\rm dev}(\bm{u})\|_\Omega
\end{equation}
$CK$ denoting the space \eqref{eq:conformal-killing-fields}. Suppose now that
\eqref{eq:continuous-trace-free-korn} fails. Then there are $\bm{u}_k$ with
$\|\nabla\bm{u}_k\|_\Omega = 1$ and
$\|\bfeps^{\rm dev}(\bm{u}_k)\|_\Omega
+\|\bm{u}_k\|_{\Gamma\suprm{D}}\to 0$. Let $\bm{w}_k\in CK$ attain
\eqref{eq:conformal-killing-quotient} for $\bm{u}_k$ and put
$\bm{r}_k := \bm{u}_k-\bm{w}_k$, so that
$\|\bm{r}_k\|_{H^1(\Omega)}\to 0$. Then
$\|\nabla\bm{w}_k\|_\Omega\to 1$ and, by the trace inequality,
$\|\bm{w}_k\|_{\Gamma\suprm{D}}\leq
\|\bm{u}_k\|_{\Gamma\suprm{D}}+C\|\bm{r}_k\|_{H^1(\Omega)}\to 0$. On the
finite-dimensional space $CK$ the expression
$(\|\nabla\bm{w}\|_\Omega^2+\|\bm{w}\|_{\Gamma\suprm{D}}^2)^{1/2}$ is a
norm, since its kernel consists of constant fields vanishing on
$\Gamma\suprm{D}$, so the sequence $\bm{w}_k$ is bounded in $CK$ and a
subsequence converges to some $\bm{w}$ with
$\|\nabla\bm{w}\|_\Omega = 1$ and $\bm{w} = \bm{0}$ on
$\Gamma\suprm{D}$. The assumption on $\Gamma\suprm{D}$ makes
$\bm{w} = \bm{0}$, a contradiction.
\end{proof}

\begin{lem}\label{lem:oswald-averaging}
Let $\bm{w}$ be piecewise linear on the mesh and let $I_h\bm{w}$ be the
continuous piecewise linear field whose value at each mesh vertex is the
average of the values of $\bm{w}$ on the elements meeting that vertex.
If the elements meeting any one vertex form a face connected set, as they
do for the meshes used here, then
\begin{equation}\label{eq:oswald-estimate}
\sum_{T}\left(h_T^{-2}\|\bm{w}-I_h\bm{w}\|_T^2
+\|\nabla(\bm{w}-I_h\bm{w})\|_T^2\right)
\lesssim \sum_{F\ {\rm interior}}h_F^{-1}\|\jump{\bm{w}}\|_F^2
\end{equation}
with a constant depending only on the shape regularity; cf.\
\cite{KaPa03} for the general piecewise polynomial case.
\end{lem}

\begin{proof}
On $T$ the error $\bm{w}-I_h\bm{w}$ is linear, so equivalence of norms on
shape-regular elements gives
$\|\bm{w}-I_h\bm{w}\|_T^2\lesssim
h_T^3\sum_{p}\vert(\bm{w}\vert_T-I_h\bm{w})(p)\vert^2$, the sum running
over the four vertices of $T$. Fix a vertex $p$ and write $\omega_p$ for
the set of elements meeting it, of cardinality bounded by the shape
regularity. By the definition of the average,
$(\bm{w}\vert_T-I_h\bm{w})(p)$ is a mean of the differences
$(\bm{w}\vert_T-\bm{w}\vert_{T'})(p)$ over $T'\in\omega_p$, and each
difference telescopes along a chain of face neighbours within $\omega_p$,
each step being the value at $p$ of the jump across an interior face
containing $p$. For a linear function on a face, equivalence of norms
again gives
$\vert\jump{\bm{w}}(p)\vert^2\lesssim h_F^{-2}\|\jump{\bm{w}}\|_F^2$.
Combining, $h_T^{-2}\|\bm{w}-I_h\bm{w}\|_T^2\lesssim
h_T\sum h_F^{-2}\|\jump{\bm{w}}\|_F^2\approx
\sum h_F^{-1}\|\jump{\bm{w}}\|_F^2$, the sums running over the interior
faces of the patches of the vertices of $T$; summing over $T$, each face
appears a bounded number of times, which is the first part of
\eqref{eq:oswald-estimate}. The gradient part follows from the inverse inequality
$\|\nabla(\bm{w}-I_h\bm{w})\|_T\lesssim h_T^{-1}\|\bm{w}-I_h\bm{w}\|_T$
for linear functions.
\end{proof}

\begin{thm}\label{thm:broken-trace-free-korn}
Under the domain, mesh-connectivity, and boundary-kernel assumptions stated
above, let $\bm{v}$ be piecewise linear on the mesh; in particular
$\bm{v}$ may be any member of $\vec{V}^h$. Then
\begin{equation}\label{eq:broken-trace-free-korn}
\|\nabla_h\bm{v}\|_\Omega^2\leq C\,\Xi(\bm{v})^2
\end{equation}
with $C$ depending only on the shape regularity of the mesh, on $\Omega$
and on $\Gamma\suprm{D}$.
\end{thm}

\begin{proof}
\emph{Step 1: elementwise similarity split.} On each $T$ the gradient
$\mathsf{G}_T := \nabla\bm{v}\vert_T$ is a constant matrix. Put
$\mathsf{S}_T := {\rm skew}\,\mathsf{G}_T
+\tfrac13({\rm tr}\,\mathsf{G}_T)\mathsf{I}$ and define the piecewise
linear field $\bm{w}$ by
$\bm{w}\vert_T(x) := \bar{\bm{v}}_T+\mathsf{S}_T(x-x_T)$, with $x_T$ the
centroid of $T$ and $\bar{\bm{v}}_T$ the mean of $\bm{v}$ over $T$. Each
$\bm{w}\vert_T$ is a similarity field, so
$\bfeps^{\rm dev}(\bm{w}) = \bm{0}$ elementwise, and
\begin{equation}
\nabla(\bm{v}-\bm{w})\vert_T
= \mathsf{G}_T-\mathsf{S}_T
= {\rm sym}\,\mathsf{G}_T-\tfrac13({\rm tr}\,\mathsf{G}_T)\mathsf{I}
= \bfeps^{\rm dev}(\bm{v})\vert_T
\label{eq:similarity-split}
\end{equation}
an exact identity, special to piecewise linears. Since
$\bm{v}-\bm{w}$ has zero mean on $T$, the Poincar\'e inequality gives
$\|\bm{v}-\bm{w}\|_T\lesssim h_T\|\bfeps^{\rm dev}(\bm{v})\|_T$, and the
discrete trace inequality on shape-regular elements then
$\|\bm{v}-\bm{w}\|_F^2\lesssim h_F^{-1}\|\bm{v}-\bm{w}\|_T^2
\lesssim h_F\|\bfeps^{\rm dev}(\bm{v})\|_T^2$ for every face
$F\subset\partial T$. Consequently, on every face,
\begin{equation}\label{eq:similarity-jump-bound}
h_F^{-1}\|\jump{\bm{w}}\|_F^2
\lesssim h_F^{-1}\|\jump{\bm{v}}\|_F^2
+\|\bfeps^{\rm dev}(\bm{v})\|_{T^+\cup T^-}^2
\end{equation}

\emph{Step 2: averaging.} Let $I_h\bm{w}$ be the vertex averaged field of
Lemma~\ref{lem:oswald-averaging}. The right-hand side of \eqref{eq:oswald-estimate} is
$\lesssim\Xi(\bm{v})^2$ by \eqref{eq:similarity-jump-bound}.

\emph{Step 3: the continuous inequality.} Lemma~\ref{lem:continuous-trace-free-korn} applied
to $I_h\bm{w}\in[H^1(\Omega)]^3$ gives
$\|\nabla I_h\bm{w}\|_\Omega\lesssim
\|\bfeps^{\rm dev}(I_h\bm{w})\|_\Omega
+\|I_h\bm{w}\|_{\Gamma\suprm{D}}$. For the first term,
$\bfeps^{\rm dev}(\bm{w}) = \bm{0}$ elementwise, so
$\|\bfeps^{\rm dev}(I_h\bm{w})\|_\Omega
= \|\bfeps^{\rm dev}_h(I_h\bm{w}-\bm{w})\|_\Omega
\leq\|\nabla_h(I_h\bm{w}-\bm{w})\|_\Omega\lesssim\Xi(\bm{v})$, by
\eqref{eq:oswald-estimate}. For the second, on each face $F\subset\Gamma\suprm{D}$,
\begin{equation}
\|I_h\bm{w}\|_F\leq
\|I_h\bm{w}-\bm{w}\|_F+\|\bm{w}-\bm{v}\|_F+\|\bm{v}\|_F
\label{eq:averaged-boundary-control}
\end{equation}
each square is $h_F$ times a quantity that sums to
$\lesssim\Xi(\bm{v})^2$: the first through the trace inequality and
\eqref{eq:oswald-estimate}, the second by Step 1, and the third because
$\|\bm{v}\|_F = \|\jump{\bm{v}}\|_F$ there; since
$h_F\lesssim{\rm diam}\,\Omega$, the sum over $\Gamma\suprm{D}$ is
$\lesssim\Xi(\bm{v})^2$.

\emph{Step 4: assembly.} By Steps 1--3,
\begin{equation}
\|\nabla_h\bm{v}\|_\Omega\leq
\|\nabla_h(\bm{v}-\bm{w})\|_\Omega
+\|\nabla_h(\bm{w}-I_h\bm{w})\|_\Omega
+\|\nabla I_h\bm{w}\|_\Omega
\lesssim \Xi(\bm{v})
\label{eq:broken-korn-assembly}
\end{equation}
\end{proof}

\begin{cor}\label{cor:three-dimensional-coercivity}
For $n=3$ and every $\gamma_1>0$ the form \eqref{eq:cr-stabilized-form} satisfies
\eqref{eq:cr-coercivity} on $\vec{V}^h$, with a constant proportional to $\mu$
and otherwise depending only on the
shape regularity, on $\Omega$, on $\Gamma\suprm{D}$ and on $\gamma_1$.
\end{cor}

\begin{proof}
Elementwise
$(\bm{\sigma}(\bm{v},0),\bfeps(\bm{v}))_T
= 2\mu\|\bfeps^{\rm dev}(\bm{v})\|_T^2$, so
$a_h(\bm{v},\bm{v})\geq 2\mu\min(1,\gamma_1)\,(\Xi(\bm{v})^2+\Xi\subrm{E}(\bm{v})^2)$,
with $\Xi\subrm{E}(\bm{v})^2$ the sum of
$h_F^{-1}\|(\bfI-\bfn\otimes\bfn)\bm{v}\|_F^2$ over the faces on
$\Gamma\suprm{E}$, while
$\|\bm{v}\|_h^2\leq\|\nabla_h\bm{v}\|_\Omega^2
+\sum_Fh_F^{-1}\|\jump{\bm{v}}\|_F^2\leq(C+1)\,(\Xi(\bm{v})^2+\Xi\subrm{E}(\bm{v})^2)$ by
Theorem~\ref{thm:broken-trace-free-korn} and
$\|\bfeps(\bm{v})\|_T\leq\|\nabla\bm{v}\|_T$.
\end{proof}

\begin{rem}
The constant in \eqref{eq:broken-trace-free-korn} is finite but, for the
boundary configurations of this paper, not small. Computing the largest
eigenvalue of $\|\nabla_h\cdot\|_\Omega^2$ relative to $\Xi(\cdot)^2$ over
all piecewise linear fields on the box of
Section~\ref{sec:numerical-results}, with $\Gamma\suprm{D}$ taken as the two
walls and the inflow, gave the lower bounds $6.4$, $11.1$, $16.1$ and
$20.7$ on the meshes of Table~\ref{tab:stokes-three-dimensional}. Over this
range they grow roughly in proportion to $1/h$, so the computation does not
exhibit the uniform bound. The maximisers concentrate in a boundary layer
against the free edges where the outflow face meets the free faces, and the
share of their gradient energy within two cells of those edges falls from
$0.81$ to $0.65$ between the $18\times6\times6$ and $36\times12\times12$
meshes, as for a layer of fixed width that the mesh resolves progressively.
A conforming piecewise linear computation of the constant in
\eqref{eq:continuous-trace-free-korn} gives $6.8$, $10.9$, $14.4$ and
$17.2$ on the same meshes, while the corresponding $H^1_0$ quotient, whose
exact value is $2$ by a Fourier argument, is approached from below, $1.989$,
$1.998$ and $1.999$ on the three coarser meshes. The coercivity of \eqref{eq:cr-stabilized-form},
though uniform in $h$, is thus weak on nearly conformal fields concentrated
at the free boundary; the computations of
Section~\ref{sec:numerical-results} show no trace of such fields.
\end{rem}

\begin{rem}
The restriction to piecewise linears is what makes
Step 1 an identity; for piecewise $H^1$ fields the elementwise kernel
includes the quadratic members of \eqref{eq:conformal-killing-fields} and
the exact similarity split is no longer available. The corresponding
piecewise $H^1$ and $H^2$ theory, with projected jump controls, is given
in \cite{WiHo24}.
The restriction to $n=3$ is
essential, not technical: in two
dimensions the pointwise kernel of the plane deviatoric
$\bfeps-\tfrac12(\nabla\cdot\bm{u})\bfI$ is infinite-dimensional, so no
inequality of the type \eqref{eq:broken-trace-free-korn} can hold for it; the two-dimensional
form of this paper keeps the three-dimensional factor $\tfrac13$, is
bounded below elementwise through \eqref{eq:stokes-deviatoric-coercivity}, and is covered by
\cite{Br04} directly, as noted in Section~\ref{sec:stokes-cavitation}. Alternatively,
the four steps above reprove that case for piecewise linears verbatim,
with $\bfeps$ in place of $\bfeps^{\rm dev}$ and the rigid motions, the
linear kernel of $\bfeps$, in place of the similarity fields, a rigid
motion vanishing on a boundary segment of positive length being zero.
\end{rem}

\paragraph{Funding.}
Peter Hansbo was supported by the Swedish Research Council under grant
2022-03908. Mats G. Larson was supported in part by the Swedish Research
Council under grants 2021-04925 and 2025-05562, the Knut and Alice Wallenberg
Foundation under grant KAW 2025.0277, and the Swedish Research Programme
Essence. The funders had no role in the design of the study; in the collection,
analysis, or interpretation of data; in the writing of the manuscript; or in
the decision to submit the manuscript for publication.

\paragraph{CRediT Authorship Contribution Statement.}
Peter Hansbo: Conceptualization, Methodology, Software, Validation, Formal
analysis, Investigation, Writing--original draft, Writing--review and editing,
Visualization, Funding acquisition. Mats G. Larson: Conceptualization,
Methodology, Software, Validation, Formal analysis, Investigation,
Writing--original draft, Writing--review and editing, Visualization, Funding
acquisition.

\paragraph{Declaration of Competing Interest.}
The authors declare that they have no known competing financial interests or
personal relationships that could have appeared to influence the work
reported in this paper.

\paragraph{Data and Code Availability.}
No external datasets were used in this study. The MATLAB source code,
reproduction drivers, and reference numerical outputs required to reproduce
the reported tables and figures are publicly available at
\url{https://github.com/mglarson1/cavitation-fem} under the BSD 3-Clause
licence. The computations reported here use version 0.3.0, commit
\texttt{ebc9521}, a copy of which is also provided as supplementary material
to the journal.

\paragraph{Use of AI Tools.}
During the preparation and revision of this manuscript, the authors used
OpenAI's GPT-5.6 Sol through ChatGPT and Codex, as well as Claude Fable 5 and
Claude Opus 5, to improve the language and readability, assist with literature
searches, and refine the manuscript. After using these tools, the authors
reviewed and edited the content as needed and take full responsibility for the
content of the published article.

\enlargethispage{\baselineskip}
\paragraph{Authors' Addresses.}\mbox{}\par
\begingroup
\fontsize{9}{10.5}\selectfont
\noindent
Peter Hansbo, Department of Mechanical Engineering, J\"onk\"oping University,
SE-55111 J\"onk\"oping, Sweden,\\
\hspace*{1em}Corresponding author: \texttt{peter.hansbo@ju.se}.\par
\noindent
Mats G. Larson, Department of Mathematics and Mathematical Statistics,
Ume{\aa} University, SE-901 87 Ume{\aa}, Sweden,\\
\hspace*{1em}\texttt{mats.larson@umu.se}.
\endgroup

\bibliographystyle{abbrv}
\bibliography{references}

\end{document}